\documentclass[12pt,a4,oneside]{amsart}
\usepackage{amsmath,amsthm,amsfonts,amssymb,bbold,thmtools,array,nicematrix, float, hyperref,mathrsfs,url}
\usepackage{cleveref}
\usepackage[margin=2.5cm]{geometry}
\usepackage[toc,page]{appendix}

\newtheorem{thm}{Theorem}[section]
\newtheorem{lem}[thm]{Lemma}
\newtheorem{dfn}[thm]{Definition}
\newtheorem{prop}[thm]{Proposition}

\newtheorem{rem}[thm]{Remark}

\newcommand{\R}{\mathbb{R}}
\newcommand{\Q}{\mathbb{Q}}
\newcommand{\Z}{\mathbb{Z}}
\newcommand{\N}{\mathbb{N}}
\newcommand{\1}{\mathbb{1}}
\newcommand{\E}{\mathscr{E}}

\DeclareMathOperator{\rank}{rank}
\DeclareMathOperator{\vspan}{span}
\DeclareMathOperator{\image}{Im}
\DeclareMathOperator{\Sym}{Sym}
\DeclareMathOperator{\lcm}{lcm}

\title{Biangular lines with angles arccos(1/5) and arccos(3/5)}
\author{Paul Tricot}
\address{Graduate School of Information Sciences, Tohoku University}
\email{tricot.paul.alain.bernard.r3@dc.tohoku.ac.jp}


\begin{document}

\begin{abstract}
The largest known biangular lines systems in dimension 7 through 20 are due to Ganzhinov and Szöllősi \cite{ganzhinov}, and the lines make angles $\arccos(\frac{1}{5})$ and $\arccos(\frac{3}{5})$. They are constructed by making use of the nice connection with integral lattices. We explore further the connection between biangular lines and integral lattices to classify the largest biangular line systems with these angles, in dimension 7 through 10. We also give a new biangular line system in dimension 15, that matches the size of the one described in \cite{ganzhinov}.
\end{abstract}

\maketitle

\section{Introduction}

A biangular line system is a set of lines in Euclidean space of dimension $d$ passing through the origin, such that any two lines are at one of two fixed angles. This is an extension of the well researched subject of equiangular lines, where only one angle is allowed instead of two. Many important results about biangular lines are due to M. Ganzhinov and F. Szöllősi \cite{ganzhinov2,ganzhinov}. In particular, they have classified the largest biangular line systems up to dimension 6. This result was achieved by searching for all possible Gram matrices.

\begin{thm}[\cite{ganzhinov}] \label{thm1-6}
    The maximal number of biangular lines in dimension $1$ through $6$ is:
    \begin{table}[H]
    \begin{center}
    \normalfont
    \begin{NiceTabular}{|c|c|c|c|c|c|c|}
        \hline
         Dimension & 1 & 2 & 3 & 4 & 5 & 6\\
         \hline
         Size & 1 & 5 & 10 & 12 & 24 & 40\\
         \hline
    \end{NiceTabular}
    \end{center}
    \end{table}
\end{thm}

In dimension 7 and higher, the maximal number of biangular lines is not known. All we have is explicit construction of biangular line systems, which give us a lower bound on the maximum.

\begin{thm}[\cite{ganzhinov}] \label{thm7-23}
    The maximal number of biangular lines in dimension $7$ through $23$ is at least:
    \begin{table}[H]
    \begin{center}
    \normalfont
    \begin{NiceTabular}{|c|c|c|c|c|c|c|c|}
        \hline
         Dimension & 7 & 8 & 9 & 10 & 11 & 12 & 13\\
         \hline
         Size &  72 & 126 & 240 & 256 & 276 & 296 & 336\\
         \hline
    \end{NiceTabular}
    \end{center}
    \begin{center}
    \normalfont
    \begin{NiceTabular}{|c|c|c|c|c|c|c|c|}
         \hline
         Dimension & 14 & 15 & 16 & 17-20 & 21 & 22 & 23\\
         \hline
         Size & 392 & 456 & 576 & 816 & 896 & 1408 & 2300\\
         \hline
    \end{NiceTabular}
    \end{center}
    \end{table}
\end{thm}

In dimension 7 through 20, the largest biangular line systems with sizes in \Cref{thm7-23} have angles $\arccos(\frac{1}{5})$ and $\arccos(\frac{3}{5})$. In addition, by restricting ourselves to these angles, a nice connection with integral lattices can be used. We use this connection to classify the largest biangular line systems with these angles, in dimension 7 through 10. Additionally, we provide a construction for a new system of $456$ biangular lines in dimension $15$.

\section{Preliminaries}
A system of lines in Euclidean space, all passing through the origin, can be represented by a set of unit vectors. The angle between two lines can then be represented by the absolute value of the inner product between their representative vectors. For $x \in \R^n$, let $N(x) := (x,x)$ be the norm of $x$, and let $S^{n-1} := \{ x \in \R^n \mid N(x) = 1\}$ denote the unit sphere in dimension $n$. For a set $X \subseteq \R^n$, let $A(X) := \{ (x,x') \mid x,x' \in X, x \ne x'\}$.

\begin{dfn}
    A set $X \subset S^{d-1}$ represents a biangular line system if $A(X) \subseteq \{ \pm\alpha, \pm\beta\}$ for some $0 \le \alpha < \beta < 1$.
\end{dfn}

For $n,k \in \N$, The Gegenbauer polynomials $Q^n_k$ are defined as:
\begin{align*}
    & Q^n_0(\lambda) = 1,\\
    & Q^n_1(\lambda) = n \lambda,\\
    & \text{For } k \ge 1, \frac{k+1}{n+2k}Q^n_{k+1}(\lambda) = \lambda Q^n_k(\lambda)-\frac{n+k-3}{n+2k-4}Q^n_{k-1}(\lambda).
\end{align*}

\begin{thm}[\cite{delsarte}] \label{thm-del}
    Let $n \in \N$ and $X \subseteq S^{n-1}$. Then for all $k \in \N$
    $$ \sum_{x,x' \in X} Q^n_k((x,x')) \ge 0$$
\end{thm}

For $X \subseteq \R^n$, let $\langle X \rangle_\Z := \{ \sum a_i x_i \mid a_i \in \Z, x_i \in X\}$ denote the lattice generated by $X$. We will call a root any vector of norm $2$. For a lattice $L$, let  $S(L) := \{ x \in L \mid N(x) = 2\}$ denote the set of its roots. A lattice $L$ is an integral lattice if $A(L) \subseteq \Z$, and a root lattice if it is integral and generated by roots.

For two lattices $L_1, L_2$, we will denote their orthogonal sum by $L_1 \oplus L_2 := \{ [x,y] \mid x \in L_1, y\in L_2\}$, where $[x,y]$ is the vector obtained by concatenation of coordinates of $x$ and $y$. We can further define for a lattice $L$ and positive integer $d$, $L^d$ to be the orthogonal sum of $d$ copies of $L$. A root lattice is said to be irreducible if it is not the orthogonal sum of two root lattices of rank at least one. Then any root lattice is an orthogonal sum of irreducible root lattices. Let $e_1, \dots, e_n$ be the standard basis of $\R^n$, and $\1$ the vector with $1$ on all coordinates. The irreducible root lattices are known (\cite{ebeling}). Up to the action of the orthogonal group, they are:
\begin{itemize}
    \item For $n \ge 1$, $A_n := \langle e_i - e_j \mid 1 \le i,j \le n+1 \rangle_\Z$,
    \item For $n \ge 4$, $D_n := \langle e_i \pm e_j \mid 1 \le i,j \le n \rangle_\Z$,
    \item $E_8 := \langle D_8, \frac{1}{2}\1 \rangle_\Z$,
    \item $E_7 := E_8 \cap x^\perp$ where $x$ is a root of $E_8$,
    \item $E_6 := E_7 \cap y^\perp$ where $y$ is a root of $E_7$.
\end{itemize}
Note that since the group of isometries of $E_8$ acts transitively on its roots (\cite{humphreys}), $E_7$ can be defined with any root $x$ of $E_8$. Similarly $E_6$ can be defined with any root of $E_7$. Furthermore, for a lattice $L \subset \R^d$ we will denote its dual lattice $L^* := \{ x \in \R^d \mid \forall y \in L, (x,y) \in \Z \}$. As an example, we compute the dual of the root lattice $D_n$, that will be used later.

\begin{lem} \label{lem-Dndual}
    Let $n$ be a positive integer. Then $D_n^* = \Z^n \cup \left(\frac{1}{2}+\Z\right)^n$.
\end{lem}
\begin{proof}
\begin{align*}
    D_n^* &= \{ x \in \R^n \mid \forall y \in D_n, (x,y) \in \Z \}\\
     &= \{ x \in \R^n \mid \forall i,j \in \{1, \dots,n\}, x_i+x_j \in \Z \text{ and } x_i-x_j \in Z \}\\
     &= \left\{ x \in \left(\frac{1}{2}\Z\right)^n \;\middle|\; \forall i,j \in \{1, \dots,n\}, x_i \in x_j+\Z \right\}\\
     &= \Z^n \cup \left(\frac{1}{2} + \Z\right)^n.
\end{align*}
\end{proof}

We also compute the dual of the root lattice $E_7$.
\begin{lem} \label{lem-E7dual}
    Consider $E_7$ defined as $E_8 \cap \frac{1}{2}\1^\perp$, then
    $$E_7^* = \left\{ \frac{\alpha}{4}\1+x \;\middle|\; \alpha \in \{0,1,2,3\}, x \in\Z^8, 2\alpha+x_1+\dots+x_8=0 \right\}.$$
\end{lem}
\begin{proof}
\begin{align*}
    E_7^* = & \left\{ x\in \R^8 \cap \1^\perp \;\middle|\; \forall y \in E_7, (x,y) \in \Z\right\}\\
    = & \left\{ x\in \R^8 \cap \1^\perp \;\middle|\; \forall y \in S(E_7), (x,y) \in \Z\right\}\\
    = & \left\{ x\in \R^8 \cap \1^\perp \;\middle|\; \begin{aligned}
        &\forall 1 \le i < j \le 8, x_i-x_j \in \Z,\\
        &\forall \sigma \in \Sym(8), \frac{1}{2}(x_{\sigma(1)}+\dots+x_{\sigma(4)}-x_{\sigma(5)}-\dots-x_{\sigma(8)}) \in \Z
    \end{aligned} \right\}.
\end{align*}
Fix $x \in E_7^*$ and $i \in \{1, \dots, 8\}$. Let $\sigma$ be a permutation of $\{1, \dots, 8\} \setminus \{i\}$, and denote $\{i_1, \dots, i_7\} =\{1, \dots, 8\} \setminus \{i\}$. Then
\begin{align*}
    2\Z \ni & \sum_{k=0}^{6}(x_i + x_{\sigma^k(i_1)} + x_{\sigma^k(i_2)} + x_{\sigma^k(i_3)} - x_{\sigma^k(i_4)} - \dots - x_{\sigma^k(i_7)})\\
    = & 7x_i - x_{i_1} - \dots - x_{i_7}\\
    = & 8x_i.
\end{align*}
Thus every coordinate of $x$ belongs to $\frac{1}{4}\Z$, so
\begin{align*}
    E_7^* = & \left\{ \frac{\alpha}{4}\1+x \;\middle|\; \begin{aligned}
        & \alpha \in \{0,1,2,3\}, x \in\Z^8, 2\alpha+x_1+\dots+x_8=0,\\
        & \forall \sigma \in \Sym(8), x_{\sigma(1)}+\dots+x_{\sigma(4)}-x_{\sigma(5)}-\dots-x_{\sigma(8)} \in 2\Z
    \end{aligned} \right\}\\
    = & \left\{ \frac{\alpha}{4}\1+x \;\middle|\; \begin{aligned}
        & \alpha \in \{0,1,2,3\}, x \in\Z^8, 2\alpha+x_1+\dots+x_8=0,\\
        & x_1 + \dots + x_8 \in 2\Z
    \end{aligned} \right\}.
\end{align*}
\end{proof}

\begin{thm}[\cite{cameron}] \label{thm-cam}
    Let $L$ be an integral lattice of rank $n$, and $X \subseteq S(L)$ such that $A(X) \subseteq \{ 0, 1\}$.
    \begin{itemize}
        \item If $n = 1$, then $|X| \le 1$,
        \item If $n = 2$, then $|X| \le 2$,
        \item If $n = 3$, then $|X| \le 4$,
        \item If $n = 6$, then $|X| \le 16$,
        \item If $n = 7$, then $|X| \le 27$,
        \item If $n = 8$, then $|X| \le 36$,
        \item If $n \ge 9$, then $|X| \le \frac{1}{2} n (n-1)$.
    \end{itemize}
\end{thm}

\begin{rem}
    In \Cref{thm-cam}, the maximal subsets for $n \in \{1,2,3\}$ correspond to subsets of $S(A_n)$. For $n \in \{6,7,8\}$ they are subsets of $S(E_n)$, and in the remaining cases they are subsets of $S(D_n)$.
\end{rem}

The Gram matrix of a set of vectors $X$ is the matrix with rows and column indexed by the elements of $X$, and with $x,y$ entry $(x,y)$. When considering a set of roots $X$ with $A(X) \subseteq \{0,1\}$, the Gram matrix of $X$ has the form $2I+A$, where $A$ is the adjacency matrix of a graph. We will say that $X$ represents the graph. Then, we have more information on the sets $X$ that reach the maximal size in \Cref{thm-cam} for $n \ge 9$. In the following theorem, $K_n$ is the complete graph on $n$ vertices, and $L(G)$ is the line graph of the graph $G$.

\begin{thm} [\cite{cameron}] \label{thm-cam2}
    Let $L$ be an integral lattice of rank $n \ge 9$, and $X \subseteq S(L)$ such that $A(X) \subseteq \{ 0, 1\}$. Assume $|X| = \frac{1}{2} n (n-1)$. Then $X$ represents $L(K_n)$.
\end{thm}

We will focus on biangular line systems with inner products $\pm\frac{1}{5}$ and $\pm\frac{3}{5}$. The largest known biangular line systems in dimension $7$ up to $20$ have these inner products. In dimension $7, 8$ and $9$ they are constructed by applying the following proposition to the root lattices $E_6$, $E_7$ and $E_8$.

\begin{prop}[Lifting, \cite{ganzhinov}] \label{prop-gha}
    Let $X \subset \mathbb{R}^{d}$ be a set of norm $2$ vectors in an integral lattice.
    Then $Y:=\Bigl\{ [\sqrt{\frac{2}{5}} x, \sqrt{\frac{1}{5}}] \mid x \in X \Bigr\}$ represents a biangular line system in $\mathbb{R}^{d+1}$ with $A(Y)=\{\pm\frac{1}{5}, \pm\frac{3}{5}\}$.
\end{prop}

By applying this transformation to the vectors of norm $2$, the set of inner products is linearly shifted from consecutive integers to $\pm\frac{1}{5}$ and $\pm\frac{3}{5}$. This is possible because these inner products separate the interval $[-1, 1]$ into segments of equal length $\frac{2}{5}$. It is also possible to use a similar method with vectors of norm $3$.

\begin{prop}[Switching root] \label{prop-sw}
    Let $Y \subset \mathbb{R}^{d+1}$ be a set of norm $3$ vectors in an integral lattice. Assume that there exists $r \in \mathbb{R}^{d+1}$ such that for all $y \in Y$ we have $( y,r ) = 1$, and $A(Y) \subseteq \{-1,0,1,2 \}$. Then $X:=\Bigl\{\sqrt{\frac{2}{5}}(y-\frac{1}{2}r) \mid y \in Y\Bigr\}$ represents a biangular line system in $\mathbb{R}^d$ with $A(X)\subseteq\{\pm\frac{1}{5}, \pm\frac{3}{5}\}$.
    
    Conversely, let $X$ be a set of unit vectors in $\R^d$ such that $A(X)\subseteq\{\pm\frac{1}{5}, \pm\frac{3}{5}\}$. Then with $Y:=\Bigl\{[\sqrt{\frac{5}{2}}x, \sqrt{\frac{1}{2}}] \mid x \in X\Bigr\}$ and $r=[0_{\mathbb{R}^d}, \sqrt{2}]$, we have $A(Y) \subseteq \{-1,0,1,2 \}$, and for all $y \in Y$ we have $( y,r ) = 1$.
\end{prop}
We will use this connection between biangular lines with inner products $\pm\frac{1}{5}$ and $\pm\frac{3}{5}$ and integral lattices to classify those biangular line systems that reach the largest size. Given a system of lines, any transformation from the orthogonal group produces another system of lines with the same mutual inner products. If two sets $X,X' \subseteq S^{d-1}$ are in the same orbit of the orthogonal group, we will say that they are equivalent and write $X \simeq X'$. Additionally, a vector $x \in S^{d-1}$ represents the same line as $-x$, so we will say that $X \simeq X'$ if $X'$ can be obtained from $X$ by replacing vectors by their opposite. Finally, if $Y:=\Bigl\{[\sqrt{\frac{5}{2}}x, \sqrt{\frac{1}{2}}] \mid x \in X\Bigr\}$ and $Y':=\Bigl\{[\sqrt{\frac{5}{2}}x, \sqrt{\frac{1}{2}}] \mid x \in X'\Bigr\}$, we will say that $Y$ and $Y'$ are equivalent whenever $X$ and $X'$ are, and write $Y \simeq Y'$. We will classify biangular line systems up to equivalence.


\section{Biangular lines and Integral lattices}

In this section, we establish some properties of biangular line systems, independent of the dimension. In the lemmas that follow, $d$ is a positive integer, and $X$ is a subset of the unit sphere $S^{d-1}$ with $A(X) \subseteq \{ \pm\frac{1}{5}, \pm\frac{3}{5}\}$. We further define $Y:=\Bigl\{[\sqrt{\frac{5}{2}}x, \sqrt{\frac{1}{2}}] \mid x \in X\Bigr\}$ and $r=[0_{\mathbb{R}^d}, \sqrt{2}]$. We will consider, for $i \in \{\pm\frac{1}{5}, \pm\frac{3}{5}\}$, $X_{i}(x_0) := \{ x \in X \cup (-X) \mid (x_0,x) = i\}$.

\begin{lem} \label{lem-z}
    Fix $x_0 \in X$, and assume that $-\frac{1}{5} \notin A(X_{\frac{3}{5}}(x_0))$. Then there exists a set of roots $Z$, such that $A(Z) \subseteq \{0,1\}$, $|Z|=|X_{\frac{3}{5}}(x_0)|$, and $\dim(\vspan(Z)) = \dim(\vspan(X_{\frac{3}{5}}(x_0),x_0))$.
\end{lem}
\begin{proof}
First, $-\frac{3}{5} \notin A(X_\frac{3}{5}(x_0))$ because the matrix $$\begin{bmatrix}
    1 & \frac{3}{5} & \frac{3}{5}\\
    \frac{3}{5} & 1 & -\frac{3}{5}\\
    \frac{3}{5} & -\frac{3}{5} & 1\\
\end{bmatrix}$$ is not positive semidefinite, so is not a Gram matrix. Therefore $A(X_\frac{3}{5}(x_0)) \subseteq \{\frac{1}{5}, \frac{3}{5}\}$.

Define $V:=x_0^\perp$. For $x \in X_\frac{3}{5}(x_0)$, the orthogonal projection $p_V$ onto $V$ is given by $p_V(x) = x-\frac{3}{5}x_0$. For $x, x' \in X_\frac{3}{5}(x_0)$ distinct,
\begin{align*}
    (p_V(x), p_V(x')) & = (x,x')-\frac{9}{25} \\
     & \in \{ -\frac{4}{25}, \frac{6}{25}\},
\end{align*}
and $p_V(X_\frac{3}{5}(x_0)) \subset \frac{4}{5}S^{k-2}$. Define $Z:= \Bigl\{[\sqrt{\frac{5}{2}}x, \sqrt{\frac{2}{5}}] \mid x \in p_V(X_\frac{3}{5}(x_0))\Bigr\}$. Then $Z$ is a set of norm 2 vectors, and $A(Z) \subseteq \{ 0,1\}$. Furthermore, we have $|X_\frac{3}{5}(x_0)| = |p_V(X_\frac{3}{5}(x_0))| = |Z|$. Lastly, $\dim(\vspan(Z)) = \dim(\vspan(p_V(X_{\frac{3}{5}}(x_0)))) - 1= \dim(\vspan(X_{\frac{3}{5}}(x_0),x_0))$.
    
\end{proof}

The next lemma gives a condition on the minimal number of roots in $L$. This will allow us to eliminate some possibilities for the root sublattice of $L$, that is to say the sublattice of $L$ generated by its roots.

\begin{lem} \label{lem-min-roots}
    Define $L:=\langle Y,r\rangle_\Z$. Then $|S(L) \cap r^\perp| \ge \frac{2}{|X|} |\left\{ (x,x')\in X^2 \;\middle|\; (x,x') = \pm\frac{3}{5} \right\}|$.
\end{lem}
\begin{proof}
Define \begin{align*}
    f : \{ (y,y') \in (Y \cup (r-Y))^2 \mid (y,y')=2\} & \longrightarrow S(L) \cap r^\perp\\
    (y,y') & \longmapsto y-y'.
\end{align*}
Then for $u \in \image(f)$, \begin{align*}
    |f^{-1}(\{u\})| &= |\{ (y,y') \in (Y \cup (r-Y))^2 \mid y-y'=u\}|\\
     & = |\{y \in Y \cup (r-Y) \mid y-u \in Y \cup (r-Y)\}|.
\end{align*}
Fix $y_0 \in Y \cup (r-Y)$ such that $y_0-u \in Y \cup (r-Y)$. We have $3=(y_0-u,y_0-u)=5-2(y_0,u)$, so $(y_0,u)=1$. Moreover, $u\in r^\perp$, so $N(r-y_0-u)=7$, and $r-y_0 \notin \left\{y \in Y \cup (r-Y) \mid y-u \in Y \cup (r-Y)\right\}$. Thus we have $|f^{-1}(\{u\})| \le |Y|$, and
\begin{align*}
    |S(L) \cap r^\perp| &\ge \frac{1}{|Y|} \left|\left\{ (y,y') \in (Y \cup (r-Y))^2 \;\middle|\; (y,y')=2\right\}\right|\\
    &=\frac{2}{|Y|} \left|\left\{ (y,y') \in Y^2 \;\middle|\; (y,y') \in \{-1,2\} \right\}\right|\\
    &=\frac{2}{|X|} \left\{ (x,x')\in X^2 \;\middle|\; (x,x') = \pm\frac{3}{5} \right\}.
\end{align*}
\end{proof}

For the last result of this section, we will use a special configuration of three lines with specific angles. We will call a special triangle a triple of vectors $x_1, x_2, x_3 \in X \cup (-X) \subset S^{d-1}$ that verifies $(x_1, x_2)=(x_2, x_3)=\frac{3}{5}$ and $(x_1, x_3)=\frac{-1}{5}$. After applying the transformation $Y:=\Bigl\{[\sqrt{\frac{5}{2}}x, \sqrt{\frac{1}{2}}] \mid x \in X\Bigr\}$, this special triangle corresponds to a configuration $y_1, y_2, y_3 \in Y \cup (r-Y) \subset \sqrt{3}S^{d}$ such that $(y_1,y_2)=(y_2,y_3)=2$ and $(y_1,y_3)=0$. We will also call such configuration a special triangle.

\begin{lem} \label{lem-pvc}
    Define $L:=\langle Y,r \rangle_\Z$, let $R:=\langle S(L) \rangle_\Z$ be the root sublattice of $L$, and define $V:=\vspan(L) \cap R^\perp$. Assume that $Y \cup (r-Y)$ contains a special triangle. Then the orthogonal projection onto $V$ is constant up to sign on $Y \cup (r-Y)$.
\end{lem}
\begin{proof}
    By the assumption, $Y \cup (r-Y)$ contains a special triangle $y_1, y_2, y_3 \in Y \cup (r-Y) \subset \sqrt{3}S^{d}$ such that $(y_1,y_2)=(y_2,y_3)=2$ and $(y_1,y_3)=0$. For $y, y' \in Y \cup (r-Y)$ such that $(y,y')=2$, $y-y'$ is a root so $y-y' \in R$, and the orthogonal projection $p_V$ onto $V$ verifies $p_V(y-y')=0$, so $p_V(y)=p_V(y')$. Similarly when $(y,y')=-1$, $y+y'-r$ is a root, so $p_V(y)=-p_V(y')$. Therefore, $p_V(y_1)=p_V(y_2)=p_V(y_3)=:P$.
    
    Fix $y\in Y \cup (r-Y) \setminus \{y_1, y_2, y_3\}$. If one of $(y,y_1)$, $(y,y_2)$ or $(y,y_3)$ belongs to $\{-1, 2\}$, then $p_V(y) = \pm P$. If not, then $(y,y_1)$, $(y,y_2)$, $(y,y_3) \in \{0,1\}$, so we have one of the following cases.

    Case 1, $[(y,y_1), (y,y_2), (y,y_3)]=[1,0,0]$. Then $x:=2y_1-2y_2+y_3-y$ is a root. Thus $0=p_V(x)=2P-2P+P-p_V(y)$, so $p_V(y)=P$.

    Case 2, $[(y,y_1), (y,y_2), (y,y_3)]=[1,0,1]$. This case can not occur since the Gram matrix of $y_1, y_2, y_3, y$ and $r$ is not positive semidefinite.

    Case 3, $[(y,y_1), (y,y_2), (y,y_3)]=[1,1,0]$. Since $\psi(y)$ and $\psi(r-y)$ represent the same line, we can replace $y$ by $r-y$, so this case is the same as Case 1.

    Case 4, $[(y,y_1), (y,y_2), (y,y_3)]=[0,0,0]$. In this case, $x:=y_1-y_2+y_3+y-r$ is a root. $p_V(x)=0$ implies $p_V(y)=-P$.

    Case 5, $[(y,y_1), (y,y_2), (y,y_3)]=[1,1,1]$. Replacing $y$ by $r-y$, we have Case 4.

    Case 6, $[(y,y_1), (y,y_2), (y,y_3)]=[0,1,0]$. This case can not occur since the Gram matrix of $y_1, y_2, y_3, y$ and $r$ is not positive semidefinite.

    Case 7, $[(y,y_1), (y,y_2), (y,y_3)]=[0,0,1]$. We can swap $y_1$ and $y_3$, so this case is the same as Case 1.

    Case 8, $[(y,y_1), (y,y_2), (y,y_3)]=[0,1,1]$. We can swap $y_1$ and $y_3$, so this case is the same as Case 3.

    In any of these cases, we have $p_V(y)=\pm P$. Thus $p_V$ is constant up to sign on $Y\cup (r-Y)$.\\
\end{proof}


\section{Biangular lines in dimension 7}

The largest known biangular line system in dimension 7 is constructed from \Cref{prop-gha}. That is $X_7:=\Bigl\{ [\sqrt{\frac{2}{5}} x, \sqrt{\frac{1}{5}}] \mid x \in S(E_6) \Bigr\}$. This set represents $72$ lines, and $A(X_7)= \{ \pm\frac{1}{5}, \pm\frac{3}{5}\}$.

\begin{lem} \label{lem-avg-d7}
    Let $X \subseteq S^6$ be such that $A(X) \subseteq \{ \pm \frac{1}{5}, \pm \frac{3}{5}\}$ and $|X| \ge 72$. Then $\frac{1}{|X|} \Big|\Bigl\{(x,x') \in X^2 \mid (x,x') \in \{\pm \frac{3}{5}\}\Bigr\}\Big| \ge \frac{141}{7}$.
\end{lem}
\begin{proof}
    Let $$a:=\frac{1}{|X|} \Big|\Bigl\{(x,x') \in X^2 \mid (x,x') \in \{\pm \frac{3}{5}\}\Bigr\}\Big|,$$ $$b:=\frac{1}{|X|} \Big|\Bigl\{(x,x') \in X^2 \mid (x,x') \in \{\pm \frac{1}{5}\}\Bigr\}\Big|.$$ We have $|X^2|=|X|+a|X|+b|X|$, so $1+a+b=|X|\ge 72$. From \Cref{thm-del} with $n=7$ and $k=2$, since $Q^7_2(x)=\frac{9}{2}(7x^2-1)$  we have 
    \begin{align*}
        0 & \le Q^7_2(1)+aQ^7_2(\frac{3}{5})+bQ^7_2(\frac{1}{5})\\
          & = 27 + \frac{171}{25} a -\frac{81}{25}b.
    \end{align*}
    Combining both inequalities gives $a \ge \frac{141}{7}$.
\end{proof}

\begin{lem}\label{lem-fc-d7}
    Let $X \subseteq S^6$ be such that $A(X) \subseteq \{ \pm \frac{1}{5}, \pm \frac{3}{5}\}$ and $|X| \ge 72$. Then $X \cup (-X)$ contains a special triangle.
\end{lem}
\begin{proof}
    Assume that $X$ contains no special triangle. From \Cref{lem-avg-d7} we have
    $$\frac{1}{|X|}\sum_{x_0 \in X} \Big|\Bigl\{y \in X \mid (y,x_0) \in \{\pm\frac{3}{5}\}\Bigr\}\Big| \ge \frac{141}{7} \approx 20.14.$$
    Consequently there exists $x_0 \in X$ such that $\Big|\Bigl\{y \in X \mid (y,x_0) \in \{\pm\frac{3}{5}\}\Bigr\}\Big| \ge 21$, so $|X_{\frac{3}{5}}(x_0)| \ge 21$. Since $X$ contains no special triangle, we have $-\frac{1}{5} \notin A(X_\frac{3}{5}(x_0))$. From \Cref{lem-z}, there exists a set of roots $Z$ such that $A(Z) \subseteq \{0,1\}$, $|Z|=|X_{\frac{3}{5}}(x_0)|\ge 21$, and $\dim(\vspan(Z)) = \dim(\vspan(X_{\frac{3}{5}}(x_0),x_0))$. In addition, we can deduce from \Cref{thm-cam} that $\dim(\vspan(Z))=7$ and that $|Z|\le 27$. Given that 
    \begin{align*}
        X \cup (-X) &= \{\pm x_0\} \cup X_\frac{3}{5}(x_0) \cup X_\frac{-3}{5}(x_0) \cup X_\frac{1}{5}(x_0) \cup X_\frac{-1}{5}(x_0)\\
        &= \{\pm x_0\} \cup X_\frac{3}{5}(x_0) \cup (-X_\frac{3}{5}(x_0)) \cup X_\frac{1}{5}(x_0) \cup (-X_\frac{1}{5}(x_0))
    \end{align*}
    where the unions are disjoint. We have
    $$ |X| = 1+|X_\frac{3}{5}(x_0)|+|X_\frac{-1}{5}(x_0)| $$
    so
    \begin{equation} \label{eqd7}
        |X| \le 1+27+|X_\frac{-1}{5}(x_0)|
    \end{equation}

    Furthermore since $\dim(\vspan(X_{\frac{3}{5}}(x_0),x_0))=7$, there exist $x_1, \dots, x_6 \in X_\frac{3}{5}(x_0)$ such that $x_0, x_1, \dots, x_6$ is a basis of $\R^7$. Let $M$ be the Gram matrix of $x_0, x_1, \dots, x_6$. Then for $i \in \{1,\dots,6\}$, $(x_0,x_i)=\frac{3}{5}$, and for $i,j\in\{1,\dots,6\}$ distinct, $(x_i,x_j)\in\{\frac{1}{5}, \frac{3}{5}\}$. Thus $M=M_\Gamma$ where $$
        M_\Gamma := \begin{bmatrix}
            1 & \frac{3}{5} & \dots & \frac{3}{5}\\
            \frac{3}{5}\\
            \dots & & \frac{1}{5}J+\frac{4}{5}I+\frac{2}{5}A_\Gamma &\\
            \frac{3}{5}
        \end{bmatrix},
    $$J is the all one matrix, and $A_\Gamma$ is the adjacency matrix of a graph on 6 vertices. For $x\in X_\frac{-1}{5}(x_0)$ let $a(x):=[(x_0,x), \dots, (x_6,x)]$. The inner product between two elements $x, x'\in X_\frac{-1}{5}(x_0)$ can be written as $(x,x')=a(x)M^{-1}a(x')^T$. Moreover, for any $x,x'\in X_\frac{-1}{5}(x_0)$ distinct and $i \in \{1,\dots,6\}$, we have $(x,x_0)=-\frac{1}{5}$ and $(x,x_i)\in\{ -\frac{3}{5}, -\frac{1}{5}, \frac{1}{5}\}$, thus $a(x) \in \{-\frac{1}{5}\} \times \{ -\frac{3}{5}, -\frac{1}{5}, \frac{1}{5}\}^6$, $a(x)M^{-1}a(x) = 1$ and $a(x)M^{-1}a(x') \in \{ \pm\frac{1}{5}, \pm\frac{3}{5}\}$. We can notice as well that if for some $i,j \in \{2,\dots,7\}$ we have $a(x)_{i} = \frac{-3}{5}$, $a(x)_{j} =\frac{1}{5}$ and $M_{i,j} = \frac{3}{5}$, then $x_{i-1}, x_{j-1}, -x$ is a special triangle. Let
    $$ B_\Gamma := \left\{ a \in \left\{-\frac{1}{5}\right\} \times \left\{ -\frac{3}{5}, -\frac{1}{5}, \frac{1}{5}\right\}^6 \;\middle|\;
        \begin{aligned}
            &aM_\Gamma^{-1}a = 1,\\
            &\forall i,j \in\{2,\dots,7\}, [a_{i}, a_{j}, (M_\Gamma)_{i, j}] \ne [\frac{-3}{5}, \frac{1}{5}, \frac{3}{5}]
        \end{aligned} \right\}.$$
    We can define a graph $G_\Gamma$ with vertices $B_\Gamma$, and edges between $a$ and $b$ when $aM_\Gamma^{-1}b \in \{ \pm\frac{1}{5}, \pm\frac{3}{5}\}$. Then $a(X_\frac{-1}{5}(x_0))$ is a clique in $G_\Gamma$. A computer search through all possible graphs $\Gamma$ on 6 vertices shows that the size of a maximal clique in $G_\Gamma$ is at most 41. Thus $|a(X_\frac{-1}{5}(x_0))| \le 41$, so $|X_\frac{-1}{5}(x_0)| \le 41$. A magma code for this algorithm is given in Appendix \ref{magma1}. Together with the inequality \eqref{eqd7} this implies $|X| \le 69$, which is a contradiction.
\end{proof}

\begin{rem}
    \Cref{lem-fc-d7} can also be proven using semidefinite programming \cite{bachoc}. One can define a semidefinite problem with variables $v(a,b,c) := |\{(x,y,z) \in X^3 \mid (x,y)=a, (x,z)=b, (y,z)=c\}|$, and additional constraints $v(\frac{3}{5},\frac{3}{5},\frac{-1}{5})=v(\frac{3}{5},\frac{-3}{5},\frac{1}{5})=v(\frac{-3}{5},\frac{-3}{5},\frac{-1}{5})=0$.
\end{rem}

\begin{lem} \label{lem-minroots-d7}
    Let $Y \subset \sqrt{3}S^7$ be such that $A(Y) \subseteq \{-1,0,1,2\}$, $|Y| \ge 72$, and let $r \in \sqrt{2}S^7$ be such that for all $y \in Y$ we have $( y,r ) =1$. Let $L:=\langle Y,r \rangle_\Z$. Then $|S(L) \cap r^\perp| \ge 42$.
\end{lem}
\begin{proof}
    This is a direct consequence of \Cref{lem-min-roots}, \Cref{lem-avg-d7}, and the fact that $|S(L) \cap r^\perp|$ is even.
\end{proof}

\begin{lem} \label{lem-rank-d7}
    Let $Y \subset \sqrt{3}S^7$ be such that $A(Y) \subseteq \{-1,0,1,2\}$, $|Y| \ge 72$, and $ r \in \sqrt{2}S^7$ such that for all $ y \in Y, ( y,r ) =1$. Let $L:=\langle Y,r \rangle_\Z$, and $R:=\langle S(L) \rangle_\Z$. Then $\rank(R) = \rank(L) = 8$.
\end{lem}

\begin{proof}
    Let $\psi(y) := \sqrt{\frac{2}{5}}(y-\frac{1}{2}r)$ be defined for $y \in Y \cup (r-Y)$, and $X:=\{ \psi(y) \mid y \in Y\}$. If $\rank(L)=d$, then by \Cref{prop-sw}, $X$ represents a biangular line system in $\mathbb{R}^{d-1}$. The largest biangular line system in dimension 6 has 40 lines (\cite{ganzhinov}), so $\rank(L)=8$.
    
    Assume that $\dim(V) \ge 1$ where $V:=\vspan(L) \cap R^\perp$. From \Cref{lem-fc-d7}, $X \cup (-X)$ contains a special triangle, so $Y$ also does. Thus from \Cref{lem-pvc}, $p_V$ is constant up to sign on $Y \cup (r-Y)$. Furthermore, from \Cref{lem-minroots-d7} we have $|S(R)\cap r^\perp| = |S(L)\cap r^\perp| \ge 42$.
    
    First consider the case $\rank(R) \le 6$. The orthogonal projection $p_R$ onto $R$ verifies $p_R(Y \cup (r-Y)) \subseteq B(R,r,N(P))$ where $$B(R,r,N(P)) := \{x \in R^* \mid N(x)=3-N(P), (x,r)=1\}.$$ This set contains vectors of norm $3-N(P)$, and $$A(p_R(Y \cup (r-Y))) \subseteq \{ -1+N(P), \pm N(P), 1\pm N(P), 2-N(P)\}.$$ Furthermore, $B(R,r,N(P))$ can be split into equivalence classes $$\overline{B}(R,r,N(P)) := \Bigl\{\{x, r-x\} \mid x \in B(R,r,N(P))\Bigr\}.$$ Then for any $y \in Y$, $p_R(y)$ is contained in a unique class. In addition, given that $y \in Y \Rightarrow r-y \notin Y$, for a class $\{x, r-x\}$ only one of $x$ and $r-x$ can be in $p_R(Y)$. Then $\overline{p_R(Y)}:=\Bigl\{ \{x,r-x \} \mid x \in p_R(Y) \Bigr\} \subseteq \overline{B}(R,r,N(P))$, and $|\overline{p_R(Y)}|=|p_R(Y)|$. Additionally, for $x,x' \in B(R,r,N(P))$, we have $x=r-x' \Leftrightarrow (x,x')=N(P)-2 $. Let $G(R,r,N(P))$ be the graph with vertex set $\overline{B}(R,r,N(P))$, and edges $\{x,r-x\} \sim \{x', r-x'\}$ when $(x,x') \in \{ -1+N(P), \pm N(P), 1\pm N(P), 2-N(P)\}$. Then $\overline{p_R(Y)}$ is a clique in $G(R,r,N(P))$.
    
    A computer search through all graphs $G(R,r,N(P))$ for all root lattices $R$ of rank at most $6$, all possible $r \in S(R)$ such that $|S(R)\cap r^\perp| \ge 42$, and all values $N(P) \in [0,3]$ shows that the size of a largest clique is always less than $36$. Thus $|\overline{p_R(Y)}| < 36$, so $|p_R(Y)| < 36$. Note that the possible norms smaller than $3$ for vectors in $R^*$ are finite, so the values to consider for $N(P)$ are finite. That is to say, $B(R,r,N(P))$ is non-empty only for a finite number of values of $N(P)$. Lastly, for $z \in p_R(Y)$, $z$ has at most two preimages by $p_R$ in $Y$, which are $z+P$ and $z-P$. Thus $|Y|\le 2 |p_R(Y)|<72$, which contradicts the assumptions of the lemma.

    A magma code for this algorithm is given in Appendix \ref{magma2}.

    Now we consider the case $\rank(R)=7$. Then $Y \cup (r-Y) \subseteq B'(R,r,N(P))$ where $$B'(R,r,N(P)):=\{[x, \pm\sqrt{N(P)}] \mid x\in R^*, N(x)=3-N(P), (x,r)=1\}.$$ In this case as well, $B'(R,r,N(P))$ can be split into classes $$\overline{B'}(R,r,N(P)) := \Bigl\{\{x, r-x\} \mid x \in B'(R,r,N(P))\Bigr\},$$ and $\overline{Y}:=\Bigl\{\{y,r-y\} \mid y \in Y \Bigr\} \subseteq \overline{B'}(R,r,N(P))$. Let $G'(R,r,N(P))$ be the graph with vertex set $\overline{B'}(R,r,N(P))$ and edges $\{y, r-y\} \sim \{y',r-y'\}$ when $(y,y') \in \{-1, 0, 1, 2\}$. Then $\overline{Y}$ is a clique in $G'(R,r,N(P))$.
    
    A computer search through all graphs $G'(R,r,N(P))$ for all root lattices $R$ of rank $7$, all possible $r \in S(R)$ such that $|S(R)\cap r^\perp| \ge 42$, and and all values $N(P) \in [0,3]$ shows that the size of a largest clique is always at most $72$. Thus $|Y|=|\overline{Y}| \le 72$, and a unique clique of size 72 exists for $R=E_6 \oplus A_1 \oplus 0$, $r\in\{0_{E_6}\} \times S(A_1) \oplus 0$, and $N(P)=\frac{1}{2}$. We have $A_1 \simeq \sqrt{2}\Z$, thus $A_1^* \simeq \sqrt{\frac{1}{2}}\Z$, and $r=[0_{E_6},\sqrt{2},0]$. Then
    \begin{align*}
        Y &\subseteq \left\{x \in (E_6 \oplus A_1)^* \times \{\pm\sqrt{\frac{1}{2}}\} \;\middle|\; N(x) = 3, (x,r)=1 \right\} \\
        &=  S(E_6^*) \times \{\sqrt{\frac{1}{2}}\} \times \{\pm\sqrt{\frac{1}{2}}\}.\\
    \end{align*}
    Given that $E_6$ is the root lattice of rank $6$ with the most roots, we have $S(E_6^*)=S(E_6)$. Thus
    $$
        Y \subseteq S(E_6) \times \{\sqrt{\frac{1}{2}}\} \times \{\pm\sqrt{\frac{1}{2}}\}.
    $$
    Here we have one of the following two cases. If there exist $y=[x,\sqrt{\frac{1}{2}},\pm\sqrt{\frac{1}{2}}] \in Y$ and $y'=[-x,\sqrt{\frac{1}{2}},\pm\sqrt{\frac{1}{2}}] \in Y$, then since $-2 \notin A(Y)$, the sign on the last coordinate of $y$ and $y'$ is the same. This implies that $y+y'-r = [0_{E_6}, 0, \sqrt{2}] \in R$, which contradicts $R=E_6 \oplus A_1 \oplus 0$. If no such $y, y'$ exist, then since $|S(E_6)|=72\le|Y|$, there must be $z=[x,\sqrt{\frac{1}{2}},\sqrt{\frac{1}{2}}] \in Y$ and $z'=[x,\sqrt{\frac{1}{2}},-\sqrt{\frac{1}{2}}] \in Y$. Consequently, $z-z'=[0_{E_6}, 0, \sqrt{2}] \in R$, and we get the same contradiction.
    
    A magma code for this algorithm is given in Appendix \ref{magma3}.
\end{proof}

\begin{thm} \label{mainthm7}
    Let $X \subset S^6$ be such that $A(X) \subseteq \{\pm\frac{1}{5}, \pm\frac{3}{5}\}$. Then $|X| \le 72$, and $|X| = 72$ only if $X\simeq X_7$.
\end{thm}

\begin{proof}
    Let $X \subset S^6$ be such that $A(X) \subseteq \{\pm\frac{1}{5}, \pm\frac{3}{5}\}$, and assume that $|X| \ge 72$. Define $Y=\{ \phi(x) \mid x \in X\}$ where $\phi : x \mapsto [\sqrt{\frac{5}{2}} x, \sqrt{\frac{1}{2}}]$, and $r=[0_{\mathbb{R}^7}, \sqrt{2}]$. Then by \Cref{prop-sw}, $A(Y) \subseteq \{-1, 0, 1, 2\}$ and for $y \in Y$, $(r,y)=1$. Let $L:=\langle Y,r \rangle_\Z$, and $R:=\langle S(L)\rangle_\Z$. From \Cref{lem-rank-d7}, $\rank(L)=\rank(R)=8$. Then $Y \subseteq B(R,r)$ where $B(R,r):= \{x\in R^* \mid N(x)=3, (x,r)=1\}$. We can split $B(R,r)$ into classes $\overline{B}(R,r) := \Bigl\{ \{x, r-x \} \mid x \in B(R,r) \Bigr\}$. Then $\overline{Y}:=\Bigl\{\{y,r-y\} \mid y \in Y \Bigr\} \subseteq\overline{B}(R,r)$. Define the graph $G(R,r)$ with vertex set $\overline{B}(R,r)$ and edges $\{y,r-y\} \sim \{y',r-y'\}$ when $(y,y') \in \{-1, 0, 1, 2\}$. Then $\overline{Y}$ is a clique in $G(R,r)$. Furthermore, from \Cref{lem-minroots-d7} we have $|S(R)\cap r^\perp| \ge 42$.
    
    A computer search through all graphs $G(R,r)$ with all possible root lattices $R$ of rank 8, and all possible $r\in S(R)$ shows that the largest clique has size at most $72$. Thus $|Y|=|\overline{Y}| \le 72$. Furthermore there is only one clique of size $72$, for $R=E_6 \oplus A_1^2$ and $r \in \{0_{E_6}\} \times A_1^2$. A magma code for this algorithm is given in Appendix \ref{magma4}.

    Given that $|X_7|=72$ and $A(X_7)= \{ \pm\frac{1}{5}, \pm\frac{3}{5}\}$, $|X|=72$ only when $X\simeq X_7$.
\end{proof}


\section{Biangular lines in dimension 8}
In this section, we follow the same process as for dimension 7, to find the largest biangular line systems with inner products $\pm\frac{1}{5}$ and $\pm\frac{3}{5}$ in dimension 8. Since the proofs are similar to the case of dimension 7, we will not repeat some details. The largest known biangular system in dimension 8 is constructed using \Cref{prop-gha}, and is $X_8:=\Bigl\{ [\sqrt{\frac{2}{5}} x, \sqrt{\frac{1}{5}}] \mid x \in S(E_7) \Bigr\}$. We have $|X_8|=126$ and $A(X_8)=\{\pm\frac{1}{5}, \pm\frac{3}{5}\}$.

\begin{lem} \label{lem-avg-d8}
    Let $X \subseteq S^7$ be such that $A(X) \subseteq \{ \pm \frac{1}{5}, \pm \frac{3}{5}\}$ and $|X| \ge 126$. Then $\frac{1}{|X|} \Big|\Bigl\{(x,x') \in X^2 \mid (x,x') \in \{\pm \frac{3}{5}\}\Bigr\}\Big| \ge \frac{975}{32}$.
\end{lem}
\begin{proof}
    Let $$a:=\frac{1}{|X|} \Big|\Bigl\{(x,x') \in X^2 \mid (x,x') \in \{\pm \frac{3}{5}\}\Bigr\}\Big|,$$ $$b:=\frac{1}{|X|} \Big|\Bigl\{(x,x') \in X^2 \mid (x,x') \in \{\pm \frac{1}{5}\}\Bigr\}\Big|.$$ We have $|X^2|=|X|+a|X|+b|X|$, so $1+a+b=|X|\ge 126$. From \Cref{thm-del} we have 
    \begin{align*}
        0 & \le Q^8_2(1)+aQ^8_2(\frac{3}{5})+bQ^8_2(\frac{1}{5})\\
          & = 35 + \frac{47}{5} a -\frac{17}{5}b.
    \end{align*}
    Combining both inequalities gives $a \ge \frac{975}{32}$.
\end{proof}

\begin{lem} \label{lem-fc-d8}
    Let $X \subseteq S^7$ be such that $A(X) \subseteq \{ \pm \frac{1}{5}, \pm \frac{3}{5}\}$ and $|X| \ge 126$. Then $X \cup (-X)$ contains a special triangle.
\end{lem}
\begin{proof}
    Assume that $X$ contains no special triangle. From \Cref{lem-avg-d8} we have
    $$\frac{1}{|X|}\sum_{x_0 \in X} \Big|\Bigl\{x \in X \mid (x,x_0) \in \{\pm\frac{3}{5}\}\Bigr\}\Big| \ge \frac{975}{32} \approx 30.47.$$
    \sloppy Consequently there exists $x_0 \in X$ such that $\Big|\Bigl\{y \in X \mid (y,x_0) \in \{\pm\frac{3}{5}\}\Bigr\}\Big| \ge 31$, so $|X_{\frac{3}{5}}(x_0)| \ge 31$. From \Cref{lem-z}, we have a set of roots $Z$ with $A(Z) \subseteq \{ 0,1\}$, $\dim(\vspan(X_{\frac{3}{5}}(x_0),x_0))=\dim(\vspan(Z))$ and $|X_\frac{3}{5}(x_0)| = |Z|$. From \Cref{thm-cam} we can then deduce that $\dim(\vspan(X_{\frac{3}{5}}(x_0),x_0))=8$ and $|X_\frac{3}{5}(x_0)| \le 36$. We have
    $$ |X| = 1+|X_\frac{3}{5}(x_0)|+|X_\frac{-1}{5}(x_0)| $$
    so
    \begin{equation} \label{eqd8}
        |X| \le 1+36+|X_\frac{-1}{5}(x_0)|
    \end{equation}

    Furthermore since $\dim(\vspan(X_{\frac{3}{5}}(x_0),x_0))=8$, there exists $x_1, \dots, x_7 \in X_\frac{3}{5}(x_0)$ such that $x_0, x_1, \dots, x_7$ is a basis of $\R^8$. Similarly to the proof of \Cref{lem-fc-d7}, the vectors of $X_\frac{-1}{5}(x_0)$ can be written in this basis. Let $M$ be the Gram matrix of $x_0, x_1, \dots, x_7$. Then, a computer search with all possible $M$ gives $|X_\frac{-1}{5}(x_0)| \le 78$. The magma code given in Appendix \ref{magma1} can be used for this. Together with the inequality \ref{eqd8} this implies $|X| \le 115$, which is a contradiction.
\end{proof}

\begin{lem} \label{lem-minroots-d8}
    Let $Y \subset \sqrt{3}S^8$ be such that $A(Y) \subseteq \{-1,0,1,2\}$, $|Y| \ge 126$, and let $r \in \sqrt{2}S^8$ be such that for all $y \in Y$ we have $( y,r ) =1$. Let $L:=\langle Y,r \rangle_\Z$. Then $|S(L) \cap r^\perp| \ge 62$.
\end{lem}
\begin{proof}
This is a direct consequence of \Cref{lem-min-roots}, \Cref{lem-avg-d8}, and the fact that $|S(L) \cap r^\perp|$ is even.
\end{proof}

\begin{lem} \label{lem-rank-d8}
    Let $Y \subset \sqrt{3}S^8$ be such that $A(Y) \subseteq \{-1,0,1,2\}$, $|Y| \ge 126$, and $ r \in \sqrt{2}S^8$ such that for all $ y \in Y, ( y,r ) =1$. Let $L:=\langle Y,r \rangle_\Z$, and $R:=\langle S(L) \rangle_\Z$. Then $\rank(R) = \rank(L) = 9$.
\end{lem}

\begin{proof}
    Let $\psi(y) := \sqrt{\frac{2}{5}}(y-\frac{1}{2}r)$ be defined for $y \in Y \cup (r-Y)$, and $X:=\{ \psi(y) \mid y \in Y\}$. If $\rank(L)=d$, then by \Cref{prop-sw} $X$ represents a biangular line system in $\mathbb{R}^{d-1}$ with $A(X) \subseteq \{ \pm \frac{1}{5}, \pm \frac{3}{5}\}$. Thus by \Cref{mainthm7}, $\rank(L)=9$.
    
    Assume that $\dim(V) \ge 1$ where $V:=\vspan(L) \cap R^\perp$. From \Cref{lem-fc-d8}, $X \cup (-X)$ contains a special triangle. Thus from \Cref{lem-pvc}, $p_V$ is constant up to sign on $Y \cup (r-Y)$. Furthermore, from \Cref{lem-minroots-d8} we have $|S(R)\cap r^\perp| = |S(L)\cap r^\perp| \ge 62$.
    
    First consider the case $\rank(R) \le 7$. The orthogonal projection $p_R$ onto $R$ verifies $p_R(Y \cup (r-Y)) \subseteq B(R,r,N(P))$ where $B(R,r,N(P)) := \{x \in R^* \mid N(x)=3-N(P), (x,r)=1\}$ contains vectors of norm $3-N(P)$, and $A(p_R(Y \cup (r-Y))) \subseteq \{ -1+N(P), \pm N(P), 1\pm N(P), 2-N(P)\}$. Furthermore, $B(R,r,N(P))$ can be split into equivalence classes $\overline{B}(R,r,N(P)) := \Bigl\{\{x, r-x\} \mid x \in B(R,r,N(P))\Bigr\}$. Then for any $y \in Y$, $p_R(y)$ is contained in a unique class. In addition, given that $y \in Y \Rightarrow r-y \notin Y$, for a class $\{x, r-x\}$ only one of $x$ and $r-x$ can be in $p_R(Y)$. Thus $\overline{p_R(Y)}:=\Bigl\{ \{x,r-x \} \mid x \in p_R(Y) \Bigr\} \subseteq \overline{B}(R,r,N(P))$, and $|\overline{p_R(Y)}|=|p_R(Y)|$. Additionally, for $x,x' \in B(R,r,N(P))$, we have $x=r-x' \Leftrightarrow (x,x')=N(P)-2 $. Let $G(R,r,N(P))$ be the graph with vertex set $\overline{B}(R,r,N(P))$, and edges $\{x,r-x\} \sim \{x',r-x'\}$ when $(x,x') \in \{ -1+N(P), \pm N(P), 1\pm N(P), 2-N(P)\}$. Then $\overline{p_R(Y)}$ is a clique in $G(R,r,N(P))$.
    
    A computer search through all root lattices $R$ of rank at most $7$, all possible $r \in S(R)$ such that $|S(R)\cap r^\perp| \ge 62$, and all values $N(P) \in [0,3]$ shows that $|p_R(Y)| < 63$. Lastly, for $z \in p_R(Y)$, $z$ has at most two preimages by $p_R$ in $Y$, which are $z+P$ and $z-P$. Thus $|Y|\le 2 |p_R(Y)|<126$, which contradicts the assumptions of the lemma. The algorithm in Appendix \ref{magma2} can be used for this computer searches.

    Now we consider the case $\rank(R)=8$. Then $Y \cup (r-Y) \subseteq B'(R,r,N(P))$ where $B'(R,r,N(P)):=\{[x, \pm\sqrt{N(P)}] \mid x\in R^*, N(x)=3-N(P), (x,r)=1\}$. In this case as well, $B'(R,r,N(P))$ can be split into classes $\overline{B'}(R,r,N(P)) := \Bigl\{\{x, r-x\} \mid x \in B'(R,r,N(P))\Bigr\}$. Then $\overline{Y}:=\Bigl\{\{y,r-y\} \mid y \in Y \Bigr\} \subseteq \overline{B'}(R,r,N(P))$, and $|\overline{Y}|=|Y|$. Let $G'(R,r,N(P))$ be the graph with vertex set $\overline{B'}(R,r,N(P))$ and edges $\{y,r-y\} \sim \{y', r-y'\}$ when $(y,y') \in \{-1, 0, 1, 2\}$. Then $\overline{Y}$ is a clique in $G'(R,r,N(P))$.
    
    A computer search through all root lattices $R$ of rank $8$, all possible $r \in S(R)$ such that $|S(R)\cap r^\perp| \ge 62$, and and all values $N(P) \in [0,3]$ shows that $|Y| \le 126$. Furthermore, there are exactly five cliques of size $126$. One clique is for $R=E_7 \oplus A_1 \oplus 0$, $r=[0_{E_7}, \sqrt{2},0]$ and $N(P)=\frac{5}{2}$. The other cliques are for $R=D_6 \oplus A_1^2 \oplus 0$, $r=[0_{D_6+A_1}, \sqrt{2},0]$ and $N(P)=\frac{5}{2}$. The algorithm in Appendix \ref{magma3} can be used for this computer searches.
    
    We have
    \begin{align*}
        E_7 &\simeq E_8 \cap (e_8+e_7)^\perp\\
         &= \left\{\frac{\alpha}{2}\1+x \;\middle|\; \alpha \in \{0,1\},\; x \in \Z^8,\; \sum_{i=1}^{8} x_i \in 2\Z,\; x_7+\frac{\alpha}{2} = -x_8-\frac{\alpha}{2} \right\}
    \end{align*}
    and
    \begin{align*}
        D_6 \oplus A_1 = \left\{ x\in\Z^8 \;\middle|\; \sum_{i=1}^{6} x_i \in 2\Z,\; x_7=-x_8\right\}.
    \end{align*}
    This implies $D_6 \oplus A_1 \subset E_7$, so $D_6 \oplus A_1^2 \oplus 0 \subset E_7 \oplus A_1 \oplus 0$ and $(E_7 \oplus A_1 \oplus 0)^* \subset(D_6 \oplus A_1^2 \oplus 0)^*$. Therefore
    \begin{align*}
        Y &\subseteq \left\{ x\in(D_6 \oplus A_1^2)^* \times \{\pm\sqrt{\frac{1}{2}}\}\;\middle|\; N(x)=3, (x,r)=1 \right\}\\
        &= S(D_6^* + \sqrt{\frac{1}{2}}\Z) \times\{\sqrt{\frac{1}{2}}\} \times \{\pm\sqrt{\frac{1}{2}}\}
    \end{align*}
    From \Cref{lem-Dndual}, we have $D_6^*=\Z^6 \cup \left(\frac{1}{2} + \Z\right)^6$, so $S(D_6^*)=S(D_6)$ and
    \begin{align*}
        Y \subseteq & \left\{\left[x,0,\sqrt{\frac{1}{2}}, \pm\sqrt{\frac{1}{2}}\right] \;\middle|\; x \in S(D_6) \right\}\\
            & \cup \left\{\left[x, \pm\sqrt{\frac{1}{2}}, \sqrt{\frac{1}{2}}, \pm\sqrt{\frac{1}{2}}\right] \;\middle|\; x \in D_6^*,\; N(x)=\frac{3}{2} \right\}\\
            & \cup \left\{\left[0_{D_6}, \pm2\sqrt{\frac{1}{2}}, \sqrt{\frac{1}{2}}, \pm\sqrt{\frac{1}{2}}\right]\right\}
    \end{align*} 

    The set on the right hand side has size $128$, so there exists $x \in S(D_6)$ such that $y:=\left[x,0,\sqrt{\frac{1}{2}}, \sqrt{\frac{1}{2}}\right] \in Y$ and $y':=\left[x,0,\sqrt{\frac{1}{2}}, -\sqrt{\frac{1}{2}}\right] \in Y$. Then $y-y'=[0_{D_6 \oplus A_1^2}, \sqrt{2}] \in R$, which contradicts $R=D_6 \oplus A_1^2 \oplus 0$.
\end{proof}

\begin{thm} \label{mainthm8}
    Let $X \subset S^7$ be such that $A(X) \subseteq \{\pm\frac{1}{5}, \pm\frac{3}{5}\}$. Then $|X| \le 126$, and $|X| = 126$ only if $X\simeq X_8$.
\end{thm}

\begin{proof}
    Let $X \subset S^7$ be such that $A(X) \subseteq \{\pm\frac{1}{5}, \pm\frac{3}{5}\}$, and assume that $|X| \ge 126$. Define $Y=\{ \phi(x) \mid x \in X\}$ where $\phi : x \mapsto [\sqrt{\frac{5}{2}} x, \sqrt{\frac{1}{2}}]$, and $r=[0_{\mathbb{R}^8}, \sqrt{2}]$. Then by \Cref{prop-sw}, $A(Y) \subseteq \{-1, 0, 1, 2\}$ and for $y \in Y$, $(r,y)=1$. Let $L:=\langle Y,r \rangle_\Z$, and $R:=\langle S(L)\rangle_\Z$. From \Cref{lem-rank-d8}, $\rank(L)=\rank(R)=9$. Then $Y \subseteq B(R,r)$ where $B(R,r):= \{x\in R^* \mid N(x)=3, (x,r)=1\}$. We can split $B(R,r)$ into classes $\overline{B}(R,r) := \Bigl\{ \{x, r-x \} \mid x \in B(R,r) \Bigr\}$. Then $\overline{Y}:=\Bigl\{\{y,r-y\} \mid y \in Y \Bigr\} \subseteq\overline{B}(R,r)$. Define the graph $G(R,r)$ with vertex set $\overline{B}(R,r)$ and edges $\{y,r-y\} \sim \{y',r-y'\}$ when $(y,y') \in \{-1, 0, 1, 2\}$. Then $\overline{Y}$ is a clique in $G(R,r)$. Furthermore, from \Cref{lem-minroots-d8} we have $|S(R)\cap r^\perp| \ge 62$. A computer search shows that $|Y| \le 126$, and there are three cliques of size $126$. One clique is for $R=E_7 \oplus A_1^2$ and $r=[0_{E_7 \oplus A_1}, \sqrt{2}]$. The other two cliques are for $R=D_6 \oplus A_1^3$ and $r=[0_{D_6 \oplus A_1^2}, \sqrt{2}]$.  The algorithm in Appendix \ref{magma4} can be used for these computer searches.

    From the proof of \Cref{lem-rank-d8}, $D_6 \oplus A_1 \subset E_7$ and $D_6^*= \Z^6 \cup (\frac{1}{2} + \Z)^6$. We deduce $(E_7 \oplus A_1^2)^* \subset (D_6 \oplus A_1^3)^*$. Thus 
    \begin{align*}
        Y &\subseteq \left\{ x \in D_6^* + \sqrt{\frac{1}{2}}\Z + \sqrt{\frac{1}{2}}\Z \times \{ \sqrt{\frac{1}{2}}\} \;\middle|\; N(x) = 3\right\}.
    \end{align*}
    Let us denote for a permutation $\sigma \in \Sym(n)$ and a vector $x\in\R^n$, $\sigma(x_1, \dots, x_n)$ the vector obtained by permuting the coordinates of $x$ according to $\sigma$, that is to say $[x_{\sigma(1)}, \dots, x_{\sigma(n)}]$. Then
    \begin{align*}
        Y \subseteq & \left\{\left[\sigma(\pm\frac{3}{2},\pm\frac{1}{2},0,0,0,0),0,0,\sqrt{\frac{1}{2}}\right] \;\middle|\; \sigma \in \Sym(6) \right\}\\
            & \cup \left\{\left[\sigma(\pm1,\pm1,0,0,0,0), 0, \pm\sqrt{\frac{1}{2}}, \sqrt{\frac{1}{2}}\right] \;\middle|\; \sigma \in \Sym(6) \right\}\\
            & \cup \left\{\left[\sigma(\pm1,\pm1,0,0,0,0), \pm\sqrt{\frac{1}{2}}, 0, \sqrt{\frac{1}{2}}\right] \;\middle|\; \sigma \in \Sym(6) \right\}\\
            & \cup \left\{\left[\pm\frac{1}{2}, \dots,\pm\frac{1}{2}, \pm\sqrt{\frac{1}{2}}, \pm\sqrt{\frac{1}{2}}, \sqrt{\frac{1}{2}}\right] \right\}\\
            & \cup \left\{\left[0_{D_6}, \pm2\sqrt{\frac{1}{2}}, \pm\sqrt{\frac{1}{2}}, \sqrt{\frac{1}{2}}\right]\right\}\\
            & \cup \left\{\left[0_{D_6}, \pm\sqrt{\frac{1}{2}}, \pm2\sqrt{\frac{1}{2}}, \sqrt{\frac{1}{2}}\right]\right\}.
    \end{align*} 
    Thus, by defining $y^*:=\{y,r-y\}$, we can pair vectors by inner product $-2$, and
    \begin{align*}
        \overline{Y} \subseteq & \left\{\left[\sigma(\frac{3}{2},\pm\frac{1}{2},0,0,0,0),0,0,\sqrt{\frac{1}{2}}\right]^* \;\middle|\; \sigma \in \Sym(6) \right\}\\
            & \cup \left\{\left[\sigma(\pm1,\pm1,0,0,0,0), 0, \sqrt{\frac{1}{2}}, \sqrt{\frac{1}{2}}\right]^* \;\middle|\; \sigma \in \Sym(6) \right\}\\
            & \cup \left\{\left[\sigma(\pm1,\pm1,0,0,0,0), \sqrt{\frac{1}{2}}, 0, \sqrt{\frac{1}{2}}\right]^* \;\middle|\; \sigma \in \Sym(6) \right\}\\
            & \cup \left\{\left[\pm\frac{1}{2}, \dots,\pm\frac{1}{2}, \pm\sqrt{\frac{1}{2}}, \sqrt{\frac{1}{2}}, \sqrt{\frac{1}{2}}\right]^* \right\}\\
            & \cup \left\{\left[0_{D_6}, \pm2\sqrt{\frac{1}{2}}, \sqrt{\frac{1}{2}}, \sqrt{\frac{1}{2}}\right]^*\right\}\\
            & \cup \left\{\left[0_{D_6}, \pm\sqrt{\frac{1}{2}}, 2\sqrt{\frac{1}{2}}, \sqrt{\frac{1}{2}}\right]^*\right\}.
    \end{align*}
    
    The two cliques of size 126 correspond to
    \begin{align*}
        \overline{Y_1} :=&  \left\{\left[\sigma(\pm1,\pm1,0,0,0,0), 0, \sqrt{\frac{1}{2}}, \sqrt{\frac{1}{2}}\right]^* \;\middle|\; \sigma \in \Sym(6) \right\} \\
        &\cup \left\{\left[\pm\frac{1}{2}, \dots,\pm\frac{1}{2}, \pm\sqrt{\frac{1}{2}}, \sqrt{\frac{1}{2}}, \sqrt{\frac{1}{2}}\right]^* \;\middle|\; \begin{aligned}
            &\text{even number of + in the}\\
            &\text{first 6 coordinates}
        \end{aligned}\right\}\\
        & \cup \left\{\left[0_{D_6}, \pm2\sqrt{\frac{1}{2}}, \sqrt{\frac{1}{2}}, \sqrt{\frac{1}{2}}\right]^* \right\},
    \end{align*}
    and
    \begin{align*}
        \overline{Y_2} :=&  \left\{\left[\sigma(\pm1,\pm1,0,0,0,0), 0, \sqrt{\frac{1}{2}}, \sqrt{\frac{1}{2}}\right]^* \;\middle|\; \sigma \in \Sym(6) \right\} \\
        &\cup \left\{\left[\pm\frac{1}{2}, \dots,\pm\frac{1}{2}, \pm\sqrt{\frac{1}{2}}, \sqrt{\frac{1}{2}}, \sqrt{\frac{1}{2}}\right]^* \;\middle|\; \begin{aligned}
            &\text{odd number of + in the}\\
            &\text{first 6 coordinates}
        \end{aligned}\right\}\\
        & \cup \left\{\left[0_{D_6}, \pm2\sqrt{\frac{1}{2}}, \sqrt{\frac{1}{2}}, \sqrt{\frac{1}{2}}\right]^*\right\}.
    \end{align*}

$\overline{Y_2}$ can be obtained from $\overline{Y_1}$ by negating the first coordinate, thus $Y_1 \simeq Y_2$.

Note that, with $E_7 \simeq E_8 \cap (e_8-e_7)^\perp$, we have
\begin{align*}
    Y \simeq & Y_1\\
        \simeq & \left\{\left[\sigma(\pm1,\pm1,0,0,0,0), 0, \sqrt{\frac{1}{2}}, \sqrt{\frac{1}{2}}\right] \;\middle|\; \sigma \in \Sym(6) \right\} \\
        &\cup \left\{\left[\pm\frac{1}{2}, \dots,\pm\frac{1}{2}, \pm\sqrt{\frac{1}{2}}, \sqrt{\frac{1}{2}}, \sqrt{\frac{1}{2}}\right] \;\middle|\; \begin{aligned}
            &\text{even number of + in the}\\
            &\text{first 6 coordinates}
        \end{aligned}\right\}\\
        & \cup \left\{\left[0_{D_6}, \pm2\sqrt{\frac{1}{2}}, \sqrt{\frac{1}{2}}, \sqrt{\frac{1}{2}}\right] \right\}\\
        \simeq &\left\{\left[\sigma(\pm1,\pm1,0,0,0,0), 0, 0, \sqrt{\frac{1}{2}}, \sqrt{\frac{1}{2}}\right] \;\middle|\; \sigma \in \Sym(6) \right\} \\
        &\cup \left\{\left[\pm\frac{1}{2}, \dots,\pm\frac{1}{2},\pm\left(\frac{1}{2},\frac{1}{2}\right), \sqrt{\frac{1}{2}}, \sqrt{\frac{1}{2}}\right] \;\middle|\; \begin{aligned}
            &\text{even number of + in the}\\
            &\text{first 6 coordinates}
        \end{aligned}\right\}\\
        & \cup \left\{\left[0_{D_6}, \pm\left(1,1\right), \sqrt{\frac{1}{2}}, \sqrt{\frac{1}{2}}\right]\right\}\\
        =& S(E_7) \times \{\sqrt{\frac{1}{2}}\} \times \{\sqrt{\frac{1}{2}}\}\\
        =& \phi(X_8).
\end{align*}
\end{proof}


\section{Biangular lines in dimension 9}
The largest known biangular system in dimension 9 is constructed using \Cref{prop-gha}, and is $X_9:=\Bigl\{ [\sqrt{\frac{2}{5}} x, \sqrt{\frac{1}{5}}] \mid x \in S(E_8) \Bigr\}$. We have $|X_9|=240$ and $A(X_9)=\{\pm\frac{1}{5}, \pm\frac{3}{5}\}$.

\begin{lem} \label{lem-avg-d9}
    Let $X \subseteq S^8$ be such that $A(X) \subseteq \{ \pm \frac{1}{5}, \pm \frac{3}{5}\}$ and $|X| \ge 240$. Then $\frac{1}{|X|} \Big|\Bigl\{(x,x') \in X^2 \mid (x,x') \in \{\pm \frac{3}{5}\}\Bigr\}\Big| \ge \frac{151}{3}$.
\end{lem}
\begin{proof}
    Let $$a:=\frac{1}{|X|} \Big|\Bigl\{(x,x') \in X^2 \mid (x,x') \in \{\pm \frac{3}{5}\}\Bigr\}\Big|,$$ $$b:=\frac{1}{|X|} \Big|\Bigl\{(x,x') \in X^2 \mid (x,x') \in \{\pm \frac{1}{5}\}\Bigr\}\Big|.$$ We have $|X^2|=|X|+a|X|+b|X|$, so $1+a+b=|X|\ge 240$. From \Cref{thm-del} we have 
    \begin{align*}
        0 & \le Q^9_2(1)+aQ^9_2(\frac{3}{5})+bQ^9_2(\frac{1}{5})\\
          & = 44 + \frac{308}{25} a -\frac{88}{25}b.
    \end{align*}
    Combining both inequalities gives $a \ge \frac{151}{3}$.
\end{proof}

\begin{lem} \label{lem-fc-d9}
    Let $X \subseteq S^8$ be such that $A(X) \subseteq \{ \pm \frac{1}{5}, \pm \frac{3}{5}\}$ and $|X| \ge 240$. Then $X \cup (-X)$ contains a special triangle.
\end{lem}
\begin{proof}
    Assume that there is no special triangle in $X \cup (-X)$. From \Cref{lem-avg-d9} we have
    $$\frac{1}{|X|}\sum_{x_0 \in X} \Big|\Bigl\{y \in X \mid (y,x_0) \in \{\pm\frac{3}{5}\}\Bigr\}\Big| \ge \frac{151}{3} \approx 50.33.$$
    Consequently there exists $x_0 \in X$ such that $\Big|\Bigl\{y \in X \mid (y,x_0) \in \{\pm\frac{3}{5}\}\Bigr\}\Big| \ge 51$, so $|X_{\frac{3}{5}}(x_0)| \ge 51$. Since $X \cup (-X)$ does not contain a special triangle, we have $\frac{-1}{5} \notin A(X_{\frac{3}{5}}(x_0))$. From \Cref{lem-z}, we have a set of roots $Z$ with $A(Z) \subseteq \{ 0,1\}$, $|Z|\ge 51$ and $\dim(\vspan(Z)) \le 9$. This contradicts \Cref{thm-cam}.
\end{proof}

\begin{lem} \label{lem-minroots-d9}
    Let $Y \subset \sqrt{3}S^9$ be such that $A(Y) \subseteq \{-1,0,1,2\}$, $|Y| \ge 240$, and let $r \in \sqrt{2}S^9$ be such that for all $y \in Y$ we have $( y,r ) =1$. Let $L:=\langle Y,r \rangle_\Z$. Then $|S(L) \cap r^\perp| \ge 102$.
\end{lem}
\begin{proof}
This is a direct consequence of \Cref{lem-min-roots}, \Cref{lem-avg-d9}, and the fact that $|S(L) \cap r^\perp|$ is even.
\end{proof}

\begin{lem} \label{lem-rank-d9}
    Let $Y \subset \sqrt{3}S^9$ be such that $A(Y) \subseteq \{-1,0,1,2\}$, $|Y| \ge 240$, and $ r \in \sqrt{2}S^9$ such that for all $ y \in Y, ( y,r ) =1$. Let $L:=\langle Y,r \rangle_\Z$, and $R:=\langle S(L) \rangle_\Z$. Then $\rank(R) = \rank(L) = 10$.
\end{lem}

\begin{proof}
    Let $\psi(y) := \sqrt{\frac{2}{5}}(y-\frac{1}{2}r)$ be defined for $y \in Y \cup (r-Y)$, and $X:=\{ \psi(y) \mid y \in Y\}$. If $\rank(L)=d$, then by \Cref{prop-sw} $X$ represents a biangular line system in $\mathbb{R}^{d-1}$ with $A(X) \subseteq \{ \pm \frac{1}{5}, \pm \frac{3}{5}\}$. Thus by \Cref{mainthm8}, $\rank(L)=10$.
    
    Assume that $\dim(V) \ge 1$ where $V:=\vspan(L) \cap R^\perp$. From \Cref{lem-fc-d9}, $X \cup (-X)$ contains a special triangle. Thus from \Cref{lem-pvc}, $p_V$ is constant up to sign on $Y \cup (r-Y)$. Furthermore, from \Cref{lem-minroots-d9} we have $|S(R)\cap r^\perp| = |S(L)\cap r^\perp| \ge 102$.
    
   First consider the case $\rank(R) \le 8$. The orthogonal projection $p_R$ onto $R$ verifies $p_R(Y \cup (r-Y)) \subseteq B(R,r,N(P))$ where $$B(R,r,N(P)) := \{x \in R^* \mid N(x)=3-N(P), (x,r)=1\}$$ contains vectors of norm $3-N(P)$, and $$A(p_R(Y \cup (r-Y))) \subseteq \{ -1+N(P), \pm N(P), 1\pm N(P), 2-N(P)\}.$$ Furthermore, $B(R,r,N(P))$ can be split into equivalence classes $$\overline{B}(R,r,N(P)) := \Bigl\{\{x, r-x\} \mid x \in B(R,r,N(P))\Bigr\}.$$ Then for any $y \in Y$, $p_R(y)$ is contained in a unique class. In addition, given that $y \in Y \Rightarrow r-y \notin Y$, for a class $\{x, r-x\}$ only one of $x$ and $r-x$ can be in $p_R(Y)$. Thus $\overline{p_R(Y)}:=\Bigl\{ \{x,r-x \} \mid x \in p_R(Y) \Bigr\} \subseteq \overline{B}(R,r,N(P))$, and $|\overline{p_R(Y)}|=|p_R(Y)|$. Additionally, for $x,x' \in B(R,r,N(P))$, we have $x=r-x' \Leftrightarrow (x,x')=N(P)-2 $. Let $G(R,r,N(P))$ be the graph with vertex set $\overline{B}(R,r,N(P))$, and edges $\{x,r-x\} \sim \{x',r-x'\}$ when $(x,x') \in \{ -1+N(P), \pm N(P), 1\pm N(P), 2-N(P)\}$. Then $\overline{p_R(Y)}$ is a clique in $G(R,r,N(P))$.
    
    A computer search through all root lattices $R$ of rank at most $8$, all possible $r \in S(R)$ such that $|S(R)\cap r^\perp| \ge 102$, and all values $N(P) \in [0,3]$ shows that $|p_R(Y)| < 120$. Lastly, for $z \in p_R(Y)$, $z$ has at most two preimages by $p_R$ in $Y$, which are $z+P$ and $z-P$. Thus $|Y|\le 2 |p_R(Y)|<240$, which contradicts the assumptions of the lemma. The algorithm in Appendix \ref{magma2} can be used for this computer search.

    Now we consider the case $\rank(R)=9$. Then $Y \cup (r-Y) \subseteq B'(R,r,N(P))$ where $B'(R,r,N(P)):=\{[x, \pm\sqrt{N(P)}] \mid x\in R^*, N(x)=3-N(P), (x,r)=1\}$. In this case as well, $B'(R,r,N(P))$ can be split into classes $\overline{B'}(R,r,N(P)) := \Bigl\{\{x, r-x\} \mid x \in B'(R,r,N(P))\Bigr\}$. Then $\overline{Y}:=\Bigl\{\{y,r-y\} \mid y \in Y \Bigr\} \subseteq \overline{B'}(R,r,N(P))$, and $|\overline{Y}|=|Y|$. Let $G'(R,r,N(P))$ be the graph with vertex set $\overline{B'}(R,r,N(P))$ and edges $\{y,r-y\} \sim \{y', r-y'\}$ when $(y,y') \in \{-1, 0, 1, 2\}$. Then $\overline{Y}$ is a clique in $G'(R,r,N(P))$.
    
    A computer search through all root lattices $R$ of rank $9$, all possible $r \in S(R)$ such that $|S(R)\cap r^\perp| \ge 102$, and and all values $N(P) \in [0,3]$ shows that $|Y| \le 240$. Furthermore, there are exactly four cliques of size $240$. One clique is for $R=E_8 \oplus A_1 \oplus 0$, $r=[0_{E_8}, \sqrt{2},0]$ and $N(P)=\frac{5}{2}$. Two cliques are for $R=D_8 \oplus A_1 \oplus 0$, $r=[0_{D_8}, \sqrt{2},0]$ and $N(P)=\frac{5}{2}$. The last clique is for $R=E_7 \oplus A_1^2 \oplus 0$, $r=[0_{E_7 \oplus A_1}, \sqrt{2},0]$ and $N(P)=\frac{5}{2}$. The algorithm in Appendix \ref{magma3} can be used for this computer search.
    
    Since $D_8 \subset E_8$, we have $\left(E_8\oplus A_1 \oplus 0 \right)^* \subset \left(D_8\oplus A_1 \oplus 0 \right)^*$, so we only have to consider $R=D_8 \oplus A_1 \oplus 0$ and $R=E_7 \oplus A_1^2 \oplus 0$.

    Consider the case $R=D_8 \oplus A_1 \oplus 0$. From \Cref{lem-Dndual}, $D_8^* = \Z^8 \cup \left(\frac{1}{2} + \Z\right)^8$. Thus, $S(D_8^*) = S(D_8) \cup \left\{ [\pm\frac{1}{2}, \dots, \pm\frac{1}{2}] \right\}$, and
    \begin{align*}
        Y \subseteq& \left\{ \left[x,\sqrt{\frac{1}{2}},\pm\sqrt{\frac{1}{2}}\right] \;\middle|\;  x \in S(D_8^*)\right\}\\
        =& \left\{ \left[\pm e_i \pm e_j,\sqrt{\frac{1}{2}},\pm\sqrt{\frac{1}{2}}\right] \;\middle|\; 1 \le i < j \le 8 \right\}\\
        & \cup \left\{ \left[\pm \frac{1}{2}, \dots, \pm\frac{1}{2}, \sqrt{\frac{1}{2}},\pm\sqrt{\frac{1}{2}}\right] \right\}.
    \end{align*}
    The two cliques of size $240$ correspond to
    \begin{align*}
        \overline{Y_1} :=& \left\{ \left[ e_i\pm e_j, \sqrt{\frac{1}{2}}, \pm\sqrt{\frac{1}{2}}\right]^* \; \middle|\; 1 \le i < j \le 8 \right\}\\
        & \cup \left\{ \left[\frac{1}{2}, \pm \frac{1}{2}, \dots, \pm\frac{1}{2}, \sqrt{\frac{1}{2}},\pm\sqrt{\frac{1}{2}}\right]^* \; \middle|\;\begin{aligned}
            &\text{even number of + in the}\\
            &\text{first 8 coordinates}
        \end{aligned}\right\}
    \end{align*}
    and
    \begin{align*}
        \overline{Y_2} :=& \left\{ \left[ e_i\pm e_j, \sqrt{\frac{1}{2}}, \pm\sqrt{\frac{1}{2}}\right]^* \; \middle|\; 1 \le i < j \le 8 \right\}\\
        & \cup \left\{ \left[\frac{1}{2}, \pm \frac{1}{2}, \dots, \pm\frac{1}{2}, \sqrt{\frac{1}{2}},\pm\sqrt{\frac{1}{2}}\right]^* \; \middle|\;\begin{aligned}
            &\text{odd number of + in the}\\
            &\text{first 8 coordinates}
        \end{aligned}\right\}
    \end{align*}
    If $\overline{Y}=\overline{Y_1}$ or $\overline{Y}=\overline{Y_2}$, then $y :=\left[ e_1 + e_2, \sqrt{\frac{1}{2}},\sqrt{\frac{1}{2}}\right] \in Y \cup (r-Y)$ and $y' :=\left[ e_1 + e_2, \sqrt{\frac{1}{2}},-\sqrt{\frac{1}{2}}\right] \in Y \cup (r-Y)$, so $y-y'=[0_{D_8 \oplus A_1}, \sqrt{2}] \in R$, which contradicts $R=D_8 \oplus A_1 \oplus 0$.

    Now consider the case $R=E_7 \oplus A_1^2 \oplus 0$. The lattice $E_7^*$ was computed in \Cref{lem-E7dual}, and we have
    \begin{align*}
        Y \subseteq & \left\{ \left[x,\sqrt{\frac{1}{2}},\pm\sqrt{\frac{1}{2}}\right] \;\middle|\; x \in S((E_7 \oplus A_1)^*) \right\}\\
        = & \left\{ \left[x,0,\sqrt{\frac{1}{2}},\pm\sqrt{\frac{1}{2}}\right] \;\middle|\; x \in S(E_7^*) \right\}\\
        & \cup \left\{ \left[x,\pm\sqrt{\frac{1}{2}},\sqrt{\frac{1}{2}},\pm\sqrt{\frac{1}{2}}\right] \;\middle|\; x \in E_7^*, N(x)=\frac{3}{2} \right\}\\
        & \cup \left\{ \left[0_{E_7},\pm2\sqrt{\frac{1}{2}},\sqrt{\frac{1}{2}},\pm\sqrt{\frac{1}{2}}\right] \right\}\\
        = & \left\{ \left[\pm(e_i-e_j),0,\sqrt{\frac{1}{2}},\pm\sqrt{\frac{1}{2}}\right] \;\middle|\; 1\le i < j\le 8 \right\}\\
        & \cup \left\{ \left[\sigma\left(\frac{1}{2}, \frac{1}{2}, \frac{1}{2}, \frac{1}{2},\frac{-1}{2},\frac{-1}{2},\frac{-1}{2},\frac{-1}{2}\right),0,\sqrt{\frac{1}{2}},\pm\sqrt{\frac{1}{2}}\right] \;\middle|\; \sigma \in \Sym(8) \right\}\\
        & \cup \left\{ \left[\pm\sigma\left(\frac{3}{4}, \frac{3}{4}, \frac{-1}{4}, \dots, \frac{-1}{4}\right),\pm\sqrt{\frac{1}{2}},\sqrt{\frac{1}{2}},\pm\sqrt{\frac{1}{2}}\right] \;\middle|\; \sigma \in \Sym(8) \right\}\\
        & \cup \left\{ \left[0_{E_7},\pm2\sqrt{\frac{1}{2}},\sqrt{\frac{1}{2}},\pm\sqrt{\frac{1}{2}}\right] \right\}.
    \end{align*}
    The clique of size 240 then corresponds to
    \begin{align*}
        \overline{Y} = & \left\{ \left[\pm(e_i-e_j),0,\sqrt{\frac{1}{2}},\sqrt{\frac{1}{2}}\right]^* \;\middle|\; 1\le i < j\le 8 \right\}\\
        & \cup \left\{ \left[\sigma\left(\frac{1}{2}, \frac{1}{2}, \frac{1}{2}, \frac{1}{2},\frac{-1}{2},\frac{-1}{2},\frac{-1}{2},\frac{-1}{2}\right),0,\sqrt{\frac{1}{2}},\sqrt{\frac{1}{2}}\right]^* \;\middle|\; \sigma \in \Sym(8) \right\}\\
        & \cup \left\{ \left[\pm\sigma\left(\frac{3}{4}, \frac{3}{4}, \frac{-1}{4}, \dots, \frac{-1}{4}\right),\pm\sqrt{\frac{1}{2}},\sqrt{\frac{1}{2}},\sqrt{\frac{1}{2}}\right]^* \;\middle|\; \sigma \in \Sym(8) \right\}\\
        & \cup \left\{ \left[0_{E_7},\pm2\sqrt{\frac{1}{2}},\sqrt{\frac{1}{2}},\sqrt{\frac{1}{2}}\right]^* \right\}.
    \end{align*}
    Then $R \ni \left[e_1-e_2,0,\sqrt{\frac{1}{2}},\sqrt{\frac{1}{2}}\right] + \left[-e_1+e_2,0,\sqrt{\frac{1}{2}},\sqrt{\frac{1}{2}}\right] - r = [0_{E_7\oplus A_1^2}, \sqrt{2}]$, which contradicts $R=E_7 \oplus A_1^2 \oplus 0$.
\end{proof}

\begin{thm} \label{mainthm9}
    Let $X \subset S^8$ be such that $A(X) \subseteq \{\pm\frac{1}{5}, \pm\frac{3}{5}\}$. Then $|X| \le 240$, and $|X| = 240$ only if $X\simeq X_9$.
\end{thm}
\begin{proof}
    Let $X \subset S^8$ be such that $A(X) \subseteq \{\pm\frac{1}{5}, \pm\frac{3}{5}\}$, and assume that $|X| \ge 240$. Define $Y=\{ \phi(x) \mid x \in X\}$ where $\phi : x \mapsto [\sqrt{\frac{5}{2}} x, \sqrt{\frac{1}{2}}]$, and $r=[0_{\mathbb{R}^9}, \sqrt{2}]$. Then by \Cref{prop-sw}, $A(Y) \subseteq \{-1, 0, 1, 2\}$ and for $y \in Y$, $(r,y)=1$. Let $L:=\langle Y,r \rangle_\Z$, and $R:=\langle S(L)\rangle_\Z$. From \Cref{lem-rank-d9}, $\rank(L)=\rank(R)=10$. Then $Y \subseteq B(R,r)$ where $B(R,r):= \{x\in R^* \mid N(x)=3, (x,r)=1\}$. We can split $B(R,r)$ into classes $\overline{B}(R,r) := \Bigl\{ \{x, r-x \} \mid x \in B(R,r) \Bigr\}$. Then $\overline{Y}:=\Bigl\{\{y,r-y\} \mid y \in Y \Bigr\} \subseteq\overline{B}(R,r)$. Define the graph $G(R,r)$ with vertex set $\overline{B}(R,r)$ and edges $\{y,r-y\} \sim \{y',r-y'\}$ when $(y,y') \in \{-1, 0, 1, 2\}$. Then $\overline{Y}$ is a clique in $G(R,r)$. Furthermore, from \Cref{lem-minroots-d9} we have $|S(R)\cap r^\perp| \ge 102$. A computer search shows that $|Y| \le 240$, and there are five cliques of size $240$. One clique is for $R=E_8 \oplus A_1^2$ and $r=[0_{E_8 \oplus A_1}, \sqrt{2}]$. Two cliques are for $R=D_8 \oplus A_1^2$ and $r=[0_{D_8 \oplus A_1}, \sqrt{2}]$. The other two cliques are for $R=E_7 \oplus A_1^3$ and $r=[0_{E_7 \oplus A_1^2}, \sqrt{2}]$. The algorithm in Appendix \ref{magma4} can be used for these computer searches.

    Since $D_8 \subset E_8$, we have $(E_8 \oplus A_1^2)^* \subset (D_8 \oplus A_1^2)^*$, so there are only four cliques. The cliques for $R=D_8 \oplus A_1^2$ are
    \begin{align*}
        \overline{Y_1} :=& \left\{ \left[ e_i\pm e_j, \sqrt{\frac{1}{2}}, \pm\sqrt{\frac{1}{2}}\right]^* \; \middle|\; 1 \le i < j \le 8 \right\}\\
        & \cup \left\{ \left[\frac{1}{2}, \pm \frac{1}{2}, \dots, \pm\frac{1}{2}, \sqrt{\frac{1}{2}},\pm\sqrt{\frac{1}{2}}\right]^* \; \middle|\;\begin{aligned}
            &\text{even number of + in the}\\
            &\text{first 8 coordinates}
        \end{aligned}\right\}
    \end{align*}
    and
    \begin{align*}
        \overline{Y_2} :=& \left\{ \left[ e_i\pm e_j, \sqrt{\frac{1}{2}}, \pm\sqrt{\frac{1}{2}}\right]^* \; \middle|\; 1 \le i < j \le 8 \right\}\\
        & \cup \left\{ \left[\frac{1}{2}, \pm \frac{1}{2}, \dots, \pm\frac{1}{2}, \sqrt{\frac{1}{2}},\pm\sqrt{\frac{1}{2}}\right]^* \; \middle|\;\begin{aligned}
            &\text{odd number of + in the}\\
            &\text{first 8 coordinates}
        \end{aligned}\right\}
    \end{align*}
    which were also considered in the proof of \Cref{lem-rank-d9}. Since $\overline{Y_2}$ can be obtained from $\overline{Y_1}$ by changing the sign of the 8th coordinate, we have $\overline{Y_1} \simeq \overline{Y_2}$. The cliques for $R=E_7 \oplus A_1^3$ are
    \begin{align*}
        \overline{Y_3} = & \left\{ \left[\pm(e_i-e_j),0,\sqrt{\frac{1}{2}},\sqrt{\frac{1}{2}}\right]^* \;\middle|\; 1\le i < j\le 8 \right\}\\
        & \cup \left\{ \left[\sigma\left(\frac{1}{2}, \frac{1}{2}, \frac{1}{2}, \frac{1}{2},\frac{-1}{2},\frac{-1}{2},\frac{-1}{2},\frac{-1}{2}\right),0,\sqrt{\frac{1}{2}},\sqrt{\frac{1}{2}}\right]^* \;\middle|\; \sigma \in \Sym(8) \right\}\\
        & \cup \left\{ \left[\pm\sigma\left(\frac{3}{4}, \frac{3}{4}, \frac{-1}{4}, \dots, \frac{-1}{4}\right),\pm\sqrt{\frac{1}{2}},\sqrt{\frac{1}{2}},\sqrt{\frac{1}{2}}\right]^* \;\middle|\; \sigma \in \Sym(8) \right\}\\
        & \cup \left\{ \left[0_{E_7},\pm2\sqrt{\frac{1}{2}},\sqrt{\frac{1}{2}},\sqrt{\frac{1}{2}}\right]^* \right\}
    \end{align*}
    and
    \begin{align*}
        \overline{Y_4} = & \left\{ \left[\pm(e_i-e_j),\sqrt{\frac{1}{2}},0,\sqrt{\frac{1}{2}}\right]^* \;\middle|\; 1\le i < j\le 8 \right\}\\
        & \cup \left\{ \left[\sigma\left(\frac{1}{2}, \frac{1}{2}, \frac{1}{2}, \frac{1}{2},\frac{-1}{2},\frac{-1}{2},\frac{-1}{2},\frac{-1}{2}\right),\sqrt{\frac{1}{2}},0,\sqrt{\frac{1}{2}}\right]^* \;\middle|\; \sigma \in \Sym(8) \right\}\\
        & \cup \left\{ \left[\pm\sigma\left(\frac{3}{4}, \frac{3}{4}, \frac{-1}{4}, \dots, \frac{-1}{4}\right),\sqrt{\frac{1}{2}},\pm\sqrt{\frac{1}{2}},\sqrt{\frac{1}{2}}\right]^* \;\middle|\; \sigma \in \Sym(8) \right\}\\
        & \cup \left\{ \left[0_{E_7},\sqrt{\frac{1}{2}},\pm2\sqrt{\frac{1}{2}},\sqrt{\frac{1}{2}}\right]^* \right\}.
    \end{align*}
    Since $\overline{Y_4}$ can be obtained from $\overline{Y_3}$ by swapping the 9th and 10th coordinates, we have $\overline{Y_3} \simeq \overline{Y_4}$.

    Then, let $L_1$ be the lattice generated by $Y_1$ and $r$, and $L_3$ the lattice generated by $Y_3$ and $r$. Furthermore, let $R_1$ be the root sublattice of $L_1$, and $R_3$ the root sublattice of $L_3$. One can verify by computer that $R_1=E_8 \oplus A_1^2$ and $R_3=E_8 \oplus A_1^2$. Thus, all five cliques of size 240 are equivalent.

    Given that $|X_9|=240$ and $A(X_9)= \{ \pm\frac{1}{5}, \pm\frac{3}{5}\}$, $|X|=240$ only when $X\simeq X_9$.
    
\end{proof}


\section{Biangular lines in dimension 10}

The process of the classification for dimension 10 is slightly different, since there is a biangular line system of the largest size that does not contain a special triangle. From \cite{ganzhinov}, the largest known biangular line system in dimension 10 is represented by $X_{10}:= \left\{ \frac{1}{\sqrt{10}} [\pm 1, \dots, \pm 1] \in \R^{10} \;\middle|\; \text{even number of $+$} \right\}$. We have $A(X_{10}) \subseteq \{\pm\frac{1}{5}, \pm\frac{3}{5}\}$ and $|X_{10}|=256$. Furthermore, one can verify that $X_{10} \cup (-X_{10})$ contains no special triangle.

\begin{lem} \label{lem-avg-d10}
    Let $X \subseteq S^9$ be such that $A(X) \subseteq \{ \pm \frac{1}{5}, \pm \frac{3}{5}\}$ and $|X| \ge 256$. Then $\frac{1}{|X|} \Big|\Bigl\{(x,x') \in X^2 \mid (x,x') \in \{\pm \frac{3}{5}\}\Bigr\}\Big| \ge 45$.
\end{lem}
\begin{proof}
    Let $$a:=\frac{1}{|X|} \Big|\Bigl\{(x,x') \in X^2 \mid (x,x') \in \{\pm \frac{3}{5}\}\Bigr\}\Big|,$$ $$b:=\frac{1}{|X|} \Big|\Bigl\{(x,x') \in X^2 \mid (x,x') \in \{\pm \frac{1}{5}\}\Bigr\}\Big|.$$ We have $|X^2|=|X|+a|X|+b|X|$, so $1+a+b=|X|\ge 256$. From \Cref{thm-del} we have 
    \begin{align*}
        0 & \le Q^{10}_2(1)+aQ^{10}_2(\frac{3}{5})+bQ^{10}_2(\frac{1}{5})\\
          & = 54 + \frac{78}{5} a -\frac{18}{5}b.
    \end{align*}
    Combining both inequalities gives $a \ge 45$.
\end{proof}

\begin{prop} \label{prop-fc-d10}
    Let $X \subseteq S^9$ be such that $A(X) \subseteq \{ \pm \frac{1}{5}, \pm \frac{3}{5}\}$, $|X| \ge 256$. Assume that $X \cup (-X)$ contains no special triangle. Then $X \simeq X_{10}$.
\end{prop}
\begin{proof}
    We know that $X_{10}$ satisfies the assumptions of the proposition, so it is enough to show the uniqueness. For $x \in X$ and $i \in \{ \pm\frac{1}{5}, \pm\frac{3}{5} \}$, define $X_{i}(x) := \{ x' \in X \cup (-X) \mid (x,x') = i\}$. From \Cref{lem-avg-d10} we have
    $$\frac{1}{|X|}\sum_{x \in X} \Big|\Bigl\{x' \in X \mid (x,x') \in \{\pm\frac{3}{5}\}\Bigr\}\Big| \ge 45.$$
    Thus
    \begin{equation} \label{eqn3}
        \frac{1}{|X|}\sum_{x \in X} |X_\frac{3}{5}(x)| \ge 45.
    \end{equation}
    Consequently there exists $x_0 \in X$ such that $|X_\frac{3}{5}(x_0)| \ge 45$.

    Define $V:=x_0^\perp$. For $x \in X_\frac{3}{5}(x_0)$, the orthogonal projection $p_V$ onto $V$ is given by $p_V(x) = x-\frac{3}{5}x_0$. For $x, x' \in X_\frac{3}{5}(x_0)$ distinct,
    \begin{align*}
        (p_V(x), p_V(x')) & = (x,x')-\frac{9}{25} \\
         & \in \{ -\frac{4}{25}, \frac{6}{25}\},
    \end{align*}
    and $p_V(X_\frac{3}{5}(x_0)) \subset \frac{4}{5}S^{k-2}$.
    Define $Z(x_0):= \Bigl\{[\sqrt{\frac{5}{2}}x, \sqrt{\frac{2}{5}}] \mid x \in p_V(X_\frac{3}{5}(x_0))\Bigr\}$. Then $Z(x_0)$ is a set of roots, and from \Cref{thm-cam} we have $\dim(\vspan(X_{\frac{3}{5}}(x_0),x_0))=\dim(\vspan(Z(x_0)))=10$, and $|X_\frac{3}{5}(x_0)| = |Z(x_0)| = 45$. Furthermore, from \Cref{thm-cam2} we can then deduce that $Z(x_0)$ represents $L(K_{10})$. Then from \cref{eqn3} we have
    \begin{equation*}
        \frac{1}{|X|}\sum_{x \in X \setminus \{x_0\}} |X_\frac{3}{5}(x)| \ge 45.
    \end{equation*}
    By applying the same reasoning, we can see that for any $x \in X$, we have $|X_\frac{3}{5}(x)| = 45$ and $Z(x)$ represents $L(K_{10})$.
    
    Let $A$ be the adjacency matrix of $L(K_{10})$, and $J_{45}$ be the 45 by 45 all-one matrix. Then, there exists an ordering of the elements of $Z$ such that the Gram matrix of $Z$ is $2I_{45}+A$. Thus, the Gram matrix of $p_V(X_\frac{3}{5}(x_0))$ is $\frac{2}{5}(2I_{45}+A)-\frac{5}{2}J_{45}$. Finally, let $G$ be the Gram matrix of $X_{\frac{3}{5}}(x_0)$, we have
    \begin{equation*} \label{eqn-gram10}
        G = \frac{2}{5}(2I_{45}+A+\frac{1}{2}J_{45}).
    \end{equation*}

    Furthermore since $\dim(\vspan(X_{\frac{3}{5}}(x_0),x_0))=10$, there exist $x_1, \dots, x_9 \in X_\frac{3}{5}(x_0)$ such that $x_0, x_1, \dots, x_9$ is a basis of $\R^{10}$. The vectors of $X_\frac{-1}{5}(x_0)$ can be written in this basis. Let $M$ be the Gram matrix of $x_0, x_1, \dots, x_9$. Then $M$ can be obtained from a 9 by 9 principal submatrix of $G$, by adding a row and column corresponding to $x_0$. Furthermore, for $x\in X$ let $a(x):=[(x_0,x), \dots, (x_9,x)]$. The inner product between two elements $x, x'\in X$ can be written as $(x,x')=a(x)M^{-1}a(x')^T$. Thus for any $x \in X_\frac{-1}{5}(x_0)$, we have $a(x) \in \{-\frac{1}{5}\} \times \{ -\frac{3}{5}, -\frac{1}{5}, \frac{1}{5}\}^9$. Let $A_M := \{ b \in \{-\frac{1}{5}\} \times \{ -\frac{3}{5}, -\frac{1}{5}, \frac{1}{5}\}^9 \mid bM^{-1}b^T = 1\}$. We can define a graph $G_M$ with vertices $A_M$, and $b \sim b'$ when $bM^{-1}b'^T \in \{ \pm\frac{1}{5}, \pm\frac{3}{5}\}$. Then $a(X_\frac{-1}{5}(x_0))$ is a clique in $G_M$. Since $|X|\ge 256$, we have $|X_\frac{-1}{5}(x_0)|=|X|-1-|X_\frac{3}{5}(x_0)| \ge 210$. A computer search shows that the largest size of a clique in $A_M$ is 210, and there is a unique clique of size 210. Thus $a(X_\frac{-1}{5}(x_0))$ is the unique clique of size $210$. The magma code given in Appendix \ref{magma1} can be used for this computation.
\end{proof}

\begin{prop} \label{prop-ne-d10}
    There exists no $X \subseteq S^9$ that verifies $A(X) \subseteq \{ \pm \frac{1}{5}, \pm \frac{3}{5}\}$, $|X| \ge 256$, and containing a special triangle.
\end{prop}
\begin{proof}
Assume that such $X$ exists and define $Y:=\Bigl\{[\sqrt{\frac{5}{2}}x, \sqrt{\frac{1}{2}}] \mid x \in X\Bigr\}$ and $r=[0_{\mathbb{R}^d}, \sqrt{2}]$. Then there exist $y_1, y_2, y_3 \in Y \cup (r-Y)$ such that $(y_1,y_2)=(y_2,y_3)=2$ and $(y_1,y_3)=0$. Define $L:=\langle Y,r \rangle_\Z$, $R:=\langle S(L) \rangle_\Z$, and $V:=\vspan(L) \cap R^\perp$. Then from \Cref{lem-pvc}, the orthogonal projection $p_V$ onto $V$ is constant up to sign on $Y\cup(r-Y)$. Let us denote $p_V(Y\cup(r-Y))=:\{\pm P\}$. Furthermore, from \Cref{lem-min-roots} and \Cref{lem-avg-d10}, we have $|S(R)\cap r^\perp| \ge 90$.

First consider the case where $\rank(L) > \rank(R) + 1$. We have $$p_R(Y\cup(r-Y)) \subseteq \{ x \in R^* \mid N(x) = 3-N(P), (x,r)=1 \}.$$ A similar computer search to the one in lower dimensions, through root lattices of rank at most 9, shows that $|p_R(Y\cup(r-Y))| < 256$. Thus $|Y|=\frac{1}{2}|Y\cup(r-Y)|\le|p_R(Y\cup(r-Y))|< 256$, which contradicts the assumption. The algorithm in Appendix \ref{magma2} can be used for this computer search.

Next we consider the case where $\rank(L) = \rank(R) + 1$. We have $$Y\cup(r-Y) \subseteq \left\{ [x,\pm\sqrt{N(P)}] \;\middle|\; x \in R^*, N(x) = 3-N(P), (x,r)=1 \right\}.$$ A similar computer search to the one in lower dimensions, through root lattices of rank 10, shows that $|Y| \le 256$, and there are two sets of size $256$ for $R=D_9 \oplus A_1 \oplus 0$. However, a search through the Gram matrices shows that there is no special triangle. The algorithm in Appendix \ref{magma3} can be used for this computer search.

Now we consider the case where $\rank(L) = \rank(R)$. We have $$Y\cup(r-Y) \subseteq \left\{ x \in R^* \;\middle|\; N(x) = 3, (x,r)=1 \right\}.$$ A similar computer search to the one in lower dimensions, through root lattices of rank 11, shows that $|Y| \le 256$, and there are 6 sets of size $256$. Two are for $R=D_{10} \oplus A_1$, two for $R=D_8 \oplus A_1^3$, and two for $R=D_7 \oplus A_3 \oplus A_1$. However, a search through the Gram matrices shows that there is no special triangle. The algorithm in Appendix \ref{magma4} can be used for this computer search.
\end{proof}

The computer search also shows that there is a system of 242 lines, containing a special triangle. Indeed, consider $r:=[0_{D_8 \oplus A_1^{2}},\sqrt{2}]$ and $A:= \left\{ x \in (D_8 \oplus A_1^3)^* \;\middle|\; N(x) = 3, (x,r)=1\right\}$, and consider as before the graph $G$ with vertices $A$, and edges between $x$ and $y$ when $(x,y) \in \{-1,0,1,2\}$. Then $G$ has two kinds of maximal cliques (in the sense that they are not contained in a larger clique), some of size 256 that correspond to $X_{10}$, and some of size 242.


\section{Biangular line systems in dimension 15}

In dimension higher than 10, we do not have a complete classification. However, by looking at strong maximality of known biangular line systems, we were able to find a new biangular line system in dimension 15 that ties the size of the system in \cite{ganzhinov}. It is constructed from the largest equiangular line system in dimension 7.

\begin{dfn}
    A set $X \subset S^{d-1}$ represents an equiangular line system if $A(X) \subseteq \{ \pm\alpha\}$ for some $0 \le \alpha < 1$.
\end{dfn}

The question of the largest possible equiangular line system in each dimension has been solved in low dimensions. In particular, there is a system of 28 lines that is the largest equiangular line system in dimension 7 up to 14 (see \cite{lemmens}). It is represented by $$
    \E_7 := \left\{ \frac{1}{\sqrt{24}} \sigma(-3,-3,1,1,1,1,1,1) \;\middle|\; \sigma \in \Sym(8) \right\}
$$
where for a vector $[x_1, \dots, x_d]$ and a permutation $\sigma$, $\sigma(x_1,\dots,x_d):=[x_{\sigma(1)}, \dots, x_{\sigma(d)}]$. Note that $A(\E_7) = \{ \pm\frac{1}{3}\}$ and $\E_7 \subset \R^8 \cap \1^\perp \simeq \R^7$.

When searching for largest possible biangular line system, we usually find systems that are maximal, meaning that they can not be extended with more lines. However, by embedding them in one dimension higher, it is usually possible to add more lines (at the cost of one dimension).

Consider a set $X \subset S^{d-1}$ that verifies $A(X) \subseteq \{ \pm\frac{p}{q}, \pm\frac{p'}{q'}\} \subset \Q$, and denote $\lambda :=\sqrt{\lcm(q,q')}$. Then $X \oplus 0$ represents a biangular line system in dimension $d+1$ with the same inner products as $X$. Furthermore, we have $A(\lambda X) \subset \Z$, so $\lambda X$ is contained in the integral lattice $L:= \langle \lambda X \rangle_\Z$. Define $$
Z(X,\lambda):=\left\{ \frac{1}{\lambda}[z,\pm\sqrt{\lambda^2-N(z)}] \mid z \in L^*, N(z) \le \lambda^2 \text{ and } \forall y \in \lambda X, (z,y) \in A(\lambda X) \right\}.
$$
Then $X \oplus 0$ can only be extended by vectors of $Z(X,\lambda)$ to produce a larger biangular line system. Moreover, since $Z(X,\lambda)$ is contained in the intersection of a lattice and a ball of radius $\lambda$, $Z(X,\lambda)$ is finite. Define the graph $G$ with vertices $Z(X,\lambda)$ and edges between $z_1$ and $z_2$ whenever $(z_1,z_2) \in A(X)$. Then the only sets representing biangular line systems in dimension $d+1$ containing $X \oplus 0$ are $(X \oplus 0) \cup C$ where $C$ is a clique of $G$.

With this method, if $X\subset S^{d-1}$ represents a biangular line system and $A(X) \subset \Q$, then we can check by computer what the largest biangular line system in dimension $d+1$ containing $X \oplus 0$ is. That is how the following biangular line system was found.

Let $e_1, \dots, e_7$ be the standard basis of $\R^7$, and define $$
    B_7 := \left\{ \pm e_i \;\middle|\; i \in \{1,\dots,7\} \right\},
$$
$$
    C_7 := \left\{ [\pm1,\dots,\pm1] \in \R^7 \;\middle|\; \text{even number of $+$} \right\}
$$
and $$
    X_{15} := \left( \sqrt{\frac{2}{5}} B_7 \times \sqrt{\frac{3}{5}} \E_7 \oplus 0 \right) \cup \left( \sqrt{\frac{1}{10}}C_7 \times \{0_{\R^8}\} \times \{ \sqrt{\frac{3}{10}}\}\right).
$$
\begin{prop}
    The set $X_{15}$ is contained in $S^{14}$ and represents $456$ biangular lines with $A(X_{15}) \subseteq \{\pm\frac{1}{5}, \pm\frac{3}{5}\}$.
\end{prop}
\begin{proof}
    Let $\1_8$ be the all-one vector in $\R^8$. The set $X_{15}$ is contained in $S^{15} \cap [0_{\R^7}, \1_8, 0]^\perp \simeq S^{14}$. Furthermore, $|X_{15}| = |B_7||E_7|+|C_7|=14 \cdot 28+64=456$. Finally,
    $$
        A(X_{15}) = A\left(\sqrt{\frac{2}{5}} B_7 \times \sqrt{\frac{3}{5}} \E_7\right) \cup A\left( \sqrt{\frac{1}{10}}C_7\times \{ \sqrt{\frac{3}{10}}\}\right) \cup \left\{ (x,y) \;\middle|\; x \in \sqrt{\frac{2}{5}} B_7, y \in \sqrt{\frac{1}{10}}C_7 \right\}
    $$
    where
    \begin{align*}
        A\left(\sqrt{\frac{2}{5}} B_7 \times \sqrt{\frac{3}{5}} \E_7\right) &= \left\{ \frac{2}{5}a+\frac{3}{5}b \;\middle|\; a \in A(B_7) \cup \{1\}, b \in A(\E_7) \cup \{1\}\right\}\\
        &= \left\{ \frac{2}{5}a+\frac{3}{5}b \;\middle|\; a \in \{ -1,0,1\}, b \in \{ \pm\frac{1}{3}, 1\}, [a,b] \ne [1,1]\right\}\\
        &=\left\{ \pm\frac{1}{5}, \pm\frac{3}{5}\right\},
    \end{align*}
    \begin{align*}
        A\left( \sqrt{\frac{1}{10}}C_7\times \{ \sqrt{\frac{3}{10}}\}\right) &= \frac{1}{10}A(C_7)+\frac{3}{10}\\
        &= \frac{1}{10}\left\{-5, -1, 3\right\} + \frac{3}{10}\\
        &= \left\{ \pm\frac{1}{5}, \frac{3}{5}\right\}
    \end{align*}
    and
    \begin{align*}
        \left\{ (x,y) \;\middle|\; x \in \sqrt{\frac{2}{5}} B_7, y \in \sqrt{\frac{1}{10}}C_7 \right\} &= \frac{1}{5} \left\{ (x,y) \;\middle|\; x \in B_7, y \in C_7 \right\}\\
        &= \left\{ \pm\frac{1}{5} \right\}.
    \end{align*}
\end{proof}

The biangular line system of size $456$ from \cite{ganzhinov} is described as follows. The largest equiangular line system in dimension $5$ has 10 lines, and inner product $\frac{1}{3}$. Let $\E_5$ be such that $|\E_5|=10$ and $A(\E_5)=\{\pm\frac{1}{3}\}$. Then $\E_5$ represents the line graph of the complete graph on $5$ vertices (see \cite{cao}), and the Gram matrix of $\E_5$ can be computed from the adjacency matrix of this graph. Define $$
    B_{10} := \left\{ \pm e_i \;\middle|\; i \in \{1,\dots,10\} \right\}
$$
where $e_1, \dots, e_{10}$ is the standard basis of $\R^{10}$, and $$
    C_{10} := \left\{ [1, \pm 1 \dots, \pm 1] \in \R^{10} \;\middle|\; \text{even number of $-$} \right\}.
$$
Finally, define $$
    X_{15}' := \left(\sqrt{\frac{2}{5}}B_{10} \times \sqrt{\frac{3}{5}}\E_5 \right) \cup \left( \sqrt{\frac{1}{10}}C_{10} \oplus 0 \right).
$$
\begin{prop}[\cite{ganzhinov}]
    The set $X_{15}'$ is contained in $S^{14}$ and represents $456$ biangular lines with $A(X_{15}') \subseteq \{\pm\frac{1}{5}, \pm\frac{3}{5}\}$.
\end{prop}

One can verify that $X_{15}$ and $X_{15}'$ are not equivalent by computing
$$
    \left|\left\{\{x,y\} \subset X_{15} \;\middle|\; (x,y) = \pm\frac{3}{5} \right\}\right|=8316
$$ and $$
    \left|\left\{\{x,y\} \subset X_{15}' \;\middle|\; (x,y) = \pm\frac{3}{5}\right\}\right|=8460.
$$
A magma code for this computation is given in Appendix \ref{magma5}.


\section{Conclusion}

The strategy used to classify biangular line systems with inner products $\pm\frac{1}{5}, \pm\frac{3}{5}$ in dimension $7$ through $10$ can most likely be extended to higher dimension. We did not manage to complete the classification in dimension 11 because the equivalent on \Cref{lem-avg-d10} for dimension 11 gives a lower bound that is too small. However, with more work it should be possible.

\section{Acknowledgment}

I am grateful to Prof. Akihiro Munemasa for his continuous guidance through this study.

This work was supported by JST SPRING (Grant Number JPMJSP2114).

\begin{appendices}


\section{Magma code: \Cref{lem-fc-d7}, \Cref{lem-fc-d8} and \Cref{prop-fc-d10}} \label{magma1}

The following code can be used for the computations in dimension 7, 8 and 10 by changing the value of \verb+dim+ to the dimension. To run this code, it is necessary to install the optional package containing the database of small graphs. The expected output of the algorithm is given in the following table.
\begin{table}[H]
\begin{center}
\normalfont
\begin{NiceTabular}{|c|c|c|c|}
    \hline
     \verb+dim+ & 7 & 8 & 10\\
     \hline
     Size & 41 & 78 & 210\\
     \hline
\end{NiceTabular}
\end{center}
\end{table}

\begin{verbatim}
dim:=7;

maxsize:=[];
maxsize[7]:=41;
maxsize[8]:=78;
maxsize[10]:=210;

Q:=Rationals();
s:=dim-1;
Gr:=SmallGraphDatabase(s:IncludeDisconnected:=true);
maxm:=0;
u:=Matrix([[Q!1]]);
v:=Matrix([[3/5 : _ in [1..s]]]);
I:=ScalarMatrix(s, 1);
J:=Matrix([[1 : _ in [1..s]] : _ in [1..s]]);
for G in Gr do
 M:=VerticalJoin(
  HorizontalJoin(u, v),
  HorizontalJoin(Transpose(v), 1/5*J+4/5*I+2/5*AdjacencyMatrix(G))
 );
 if not IsPositiveDefinite(M) then continue G; end if;
 Mi:=M^-1;
 X:=[];
 for a in CartesianPower({-3/5, -1/5, 1/5},s) do
  // Remove special triangles
  for i in [1..s] do
   for j in [1..s] do
    if a[i] eq -3/5 and a[j] eq 1/5 and M[i+1,j+1] eq 3/5 then
     continue a;
    end if;
   end for;
  end for;
  // Check norm 1
  A:=Vector([-1/5] cat [a[i] : i in [1..#a]]);
  if (A*Mi,A) eq 1 then Append(~X,A); end if;
 end for;
 Gg:=Graph<#X | {{i,j} : i in [1..#X], j in [1..#X]
  | i lt j and (X[i]*Mi,X[j]) in {-3/5,-1/5,1/5,3/5}}>;
 maxm:=Max(maxm,#MaximumClique(Gg));
end for;
maxsize[dim] eq maxm;
\end{verbatim}


\section{Magma code: \Cref{lem-rank-d7}, \Cref{lem-rank-d8}, \Cref{lem-rank-d9} and \Cref{prop-ne-d10}} \label{magma2}
The following code can be used for the computations in dimension 7 through 10 by changing the value of \verb+dim+ to the dimension. The lower bounds on $|S(L) \cap r^\perp|$ are given in the following table.
\begin{table}[H]
\begin{center}
\normalfont
\begin{NiceTabular}{|c|c|c|c|c|}
    \hline
     \verb+dim+ & 7 & 8 & 9 & 10\\
     \hline
     bound & 42 & 62 & 102 & 90\\
     \hline
\end{NiceTabular}
\end{center}
\end{table}

\begin{verbatim}
dim:=7; // Dimension

// Size of the largest known biangular line system
maxsize:=[]; 
maxsize[7]:=72;
maxsize[8]:=126;
maxsize[9]:=240;
maxsize[10]:=256;
// Lower bound on the number of roots of L orthogonal to r
minRoots:=[];
minRoots[7]:=42;
minRoots[8]:=62;
minRoots[9]:=102;
minRoots[10]:=90;

// Compute the list of irreducible root lattices of rank at most n
function irreducibleRL(n)
 return [* i le 3 select 
  [*[*"A" cat IntegerToString(i),Lattice("A",i)*]*]
 else
  i in {6,7,8} select
   [*[*"A" cat IntegerToString(i),Lattice("A",i)*],
   [*"D" cat IntegerToString(i),Lattice("D",i)*],
   [*"E" cat IntegerToString(i),Lattice("E",i)*]*]
  else 
    [*[*"A" cat IntegerToString(i),Lattice("A",i)*],
    [*"D" cat IntegerToString(i),Lattice("D",i)*]*]
  : i in [1..n] *];
end function;

// Compute the list of root lattices of rank n,
// with representatives of roots for each lattice
function rootL (n) 
 LR:=irreducibleRL(n);
 Lats:=[* *];
 for pt in Partitions(n) do
  for li in CartesianProduct([{1..#LR[pt[i]]} : i in [1..#pt]]) do
    // Eliminate duplicates, e.g. A4+D4 and D4+A4
   if &or[ pt[i] eq pt[i+1] and li[i] lt li[i+1] : i in [1..#pt-1]] then
    continue li;
   end if;
   LatList:=[LR[pt[i]][li[i]] : i in [1..#li]];
   Latt:=DirectSum([pair[2] : pair in LatList]);
   zeros:=[*Zero(LatList[i][2]) : i in [1..#LatList]*];
   roots:=[*ShortestVectors(LatList[i][2])[1] : i in [1..#LatList]*];
   rootsMat:=[[*i eq j select roots[i] else zeros[i] : i in [1..#LatList]*]
    : j in [1..#LatList] | not LatList[j][1] in {LatList[i][1] : i in [1..j-1]}];
   rootsRep:=[Coordelt(Latt,&cat([Coordinates(x) : x in row])) : row in rootsMat];
   namesSet:={* pair[1] : pair in LatList *};
   namesList:=[ Multiplicity(namesSet,comp) eq 1 select comp
    else comp cat "^" cat IntegerToString(Multiplicity(namesSet,comp))
     : comp in Set(namesSet)];
   name:= #namesList eq 1 select namesList[1] else &cat[namesList[i] cat "+"
    : i in [1..#namesList-1]] cat namesList[#namesList];
   Append(~Lats,[*name, Latt, rootsRep*]);
  end for;
 end for;
 return Lats;
end function;

// Operations for Rank(R) < dim
for Latt in &cat[rootL(i) : i in [2..dim-1]] do
 name:=Latt[1]; R:=Latt[2]; rootsRep:=Latt[3];
 R2:= S cat [-s : s in S] where S is [pair[1] : pair in ShortVectors(R,2,2)];
 Rs:= Dual(R: Rescale:=false);
 Rsl3:= S cat [-s : s in S] where S is [pair[1]
  : pair in ShortVectors(Rs,3) | Norm(pair[1]) ne 3];
 for r in rootsRep do
  if #[x : x in R2 | (x,r) eq 0] lt minRoots[dim] then
   continue r;
  end if;
  A:=[x : x in Rsl3 | (x,r) eq 1];
  Na:={Norm(x) : x in A};
  for a in Na do
   Aa:=[x : x in A | Norm(x) eq a];
   Ia:={-3+a, -2+a, -1+a, 2-a, 3-a, 4-a};
   Gp:=Graph<#Aa | {{i,j} : i in [1..#Aa], j in [1..#Aa]
    | i lt j and (Aa[i],Aa[j]) eq 1-a}>;
   Ba:=[ Aa[Index(Random(c))] : c in Components(Gp)];
   G:=Graph<#Ba | {{i,j} : i in [1..#Ba], j in [1..#Ba]
    | i lt j and (Ba[i],Ba[j]) in Ia}>;
   // Check degree of vertices before computing maximal clique
   ds:=DegreeSequence(G);
   if #ds lt maxsize[dim]/2 or not &+[ds[i] : i in [maxsize[dim]/2..#ds]]
    ge maxsize[dim]/2 then
    continue a;
   end if;
   mc:=#MaximumClique(G);
   if mc ge maxsize[dim]/2 then
    printf "R=%o, there are %o maximum clique of size %o\n",
    name,#AllCliques(G,mc),mc;
   end if;
  end for;
 end for;
end for;
\end{verbatim}


\section{Magma code: \Cref{lem-rank-d7}, \Cref{lem-rank-d8}, \Cref{lem-rank-d9} and \Cref{prop-ne-d10}} \label{magma3}
The following code can be used for the computations in dimension 7 through 10. The code of Appendix \ref{magma2} should be run before this one.
\begin{verbatim}

// Operations for Rank(R) = dim
for Latt in rootL(dim) do
 name:=Latt[1]; R:=Latt[2]; rootsRep:=Latt[3];
 R2:= S cat [-s : s in S] where S is [pair[1] : pair in ShortVectors(R,2,2)];
 Rs:= Dual(R: Rescale:=false);
 Rsl3:= S cat [-s : s in S] where S is [pair[1]
  : pair in ShortVectors(Rs,3) | Norm(pair[1]) ne 3];
 for r in rootsRep do
  if #[x : x in R2 | (x,r) eq 0] lt minRoots[dim] then
   continue r;
  end if;
  A:=[x : x in Rsl3 | (x,r) eq 1];
  Na:={Norm(x) : x in A};
  for a in Na do
   Aa:=[x : x in A | Norm(x) eq a];
   Ba:=[[*x,s*] : x in Aa, s in {+1, -1}];
   Gp:=Graph<#Ba | {{i,j} : i in [1..#Ba], j in [1..#Ba]
    | i lt j and (Ba[i][1],Ba[j][1])+(3-a)*Ba[i][2]*Ba[j][2] eq -2}>;
   Ca:=[ Ba[Index(Random(c))] : c in Components(Gp)];
   G:=Graph<#Ca | {{i,j} : i in [1..#Ca], j in [1..#Ca]
    | i lt j and (Ca[i][1],Ca[j][1])+(3-a)*Ca[i][2]*Ca[j][2] in {-1,0,1,2}}>;
   // Check degree of vertices before computing maximal clique
   ds:=DegreeSequence(G);
   if #ds lt maxsize[dim] or not &+[ds[i] : i in [maxsize[dim]..#ds]]
    ge maxsize[dim] then
    continue a;
   end if;
   mc:=#MaximumClique(G);
   if mc ge maxsize[dim] then
    printf "R=%o, there are %o maximum clique of size %o\n",
    name,#AllCliques(G,mc),mc;
   end if;
  end for;
 end for;
end for;
\end{verbatim}


\section{Magma code: \Cref{mainthm7}, \Cref{mainthm8}, \Cref{mainthm9} and \Cref{prop-ne-d10}} \label{magma4}
The following code can be used for the computations in dimension 7 through 10. The code of Appendix \ref{magma2} should be run before this one.
\begin{verbatim}

// Operations for Rank(R)=dim+1
for Latt in rootL(dim+1) do
 name:=Latt[1]; R:=Latt[2]; rootsRep:=Latt[3];
 R2:= S cat [-s : s in S] where S is [pair[1] : pair in ShortVectors(R,2,2)];
 Rs:= Dual(R: Rescale:=false);
 Rs3:= S cat [-s : s in S] where S is [pair[1]
  : pair in ShortVectors(Rs,3,3)];
 for r in rootsRep do
  if #[x : x in R2 | (x,r) eq 0] lt minRoots[dim] then
   continue r;
  end if;
  A:=[x : x in Rs3 | (x,r) eq 1];
  Gp:=Graph<#A | {{i,j} : i in [1..#A], j in [1..#A]
   | i lt j and (A[i],A[j]) eq -2}>;
  B:=[ A[Index(Random(c))] : c in Components(Gp)];
  G:=Graph<#B | {{i,j} : i in [1..#B], j in [1..#B]
   | i lt j and (B[i],B[j]) in {-1,0,1,2}}>;
  // Check degree of vertices before computing maximal clique
  ds:=DegreeSequence(G);
  if #ds lt maxsize[dim] or not &+[ds[i] : i in [maxsize[dim]..#ds]]
   ge maxsize[dim] then
   continue r;
  end if;
  mc:=#MaximumClique(G);
  if mc ge maxsize[dim] then
    printf "R=%o, there are %o maximum clique of size %o\n",
    name,#AllCliques(G,mc),mc;
  end if;
 end for;
end for;
\end{verbatim}


\section{Magma code: Inequivalence of $X_{15}$ and $X_{15}'$}
\label{magma5}
\begin{verbatim}
Q:=Rationals();
// X15
v:=[Q!-3,-3,1,1,1,1,1,1];
E7:=[Vector(vec) : vec in {[v[i^p] : i in [1..8]] : p in Sym(8)}];
X7:=[Vector([1/2*x[i] : i in [1..7]])
 : x in CartesianPower({Q!-1,1},7) | #{i : i in [1..7] | x[i] eq 1} mod 2 eq 0];
B7:=[Vector([i eq j select s else 0 : j in [1..7]]) : i in [1..7], s in [Q!-1,1]];
X15:=[* [*a,b,Vector([0])*] : a in B7, b in E7 *]
 cat [*[*a,Vector([Q!0,0,0,0,0,0,0,0]),Vector([1])*] : a in X7 *];
G15:=Matrix([[2/5*(x[1],y[1])+1/40*(x[2],y[2])+3/10*(x[3],y[3])
 : x in X15] : y in X15]);
#[ [i,j] : i in [1..Nrows(G15)], j in [1..Ncols(G15)]
 | G15[i,j] in {-3/5, 3/5} and i lt j] eq 8316;

// X15'
G5:= Matrix([[i eq j select Q!3 else -1 : j in [1..10]]
 : i in [1..10]])+2*AdjacencyMatrix(LineGraph(CompleteGraph(5)));
X10:=[[x[i] : i in [1..10]] : x in CartesianPower({Q!-1,1},10)
 | x[1] eq 1 and #{i : i in [1..10] | x[i] eq -1} mod 2 eq 0];
B10:=[[i eq j select s else 0 : j in [1..10]] : i in [1..10], s in [Q!-1,1]];
X15p:=[[*0,Vector(x)*] : x in X10] cat [[*i,Vector(x)*]: i in [1..10], x in B10];
G15p:=1/5*Matrix([[i le #X10 select ( j le #X10 select (X15p[i][2],X15p[j][2])/2
 else (X15p[i][2],X15p[j][2])) else (j le #X10 select (X15p[i][2],X15p[j][2])
 else G5[X15p[i][1],X15p[j][1]]+2*(X15p[i][2],X15p[j][2]))
 : j in [1..#X15p]] : i in [1..#X15p]]);
#[ [i,j] : i in [1..Nrows(G15p)], j in [1..Ncols(G15p)]
 | G15p[i,j] in {-3/5, 3/5} and i lt j] eq 8460;

\end{verbatim}

\end{appendices}
 
\bibliographystyle{plain}
\bibliography{dpgl}

@article {delsarte,
    AUTHOR = {Delsarte, P. and Goethals, J. M. and Seidel, J. J.},
     TITLE = {Spherical codes and designs},
   JOURNAL = {Geometriae Dedicata},
  FJOURNAL = {Geometriae Dedicata},
    VOLUME = {6},
      YEAR = {1977},
    NUMBER = {3},
     PAGES = {363--388},
   MRCLASS = {05B99},
  MRNUMBER = {485471},
MRREVIEWER = {Michel\ Deza},
       DOI = {10.1007/bf03187604},
       URL = {https://doi.org/10.1007/bf03187604},
}

@article {cameron,
    AUTHOR = {Cameron, P. J. and Goethals, J.-M. and Seidel, J. J. and
              Shult, E. E.},
     TITLE = {Line graphs, root systems, and elliptic geometry},
   JOURNAL = {J. Algebra},
  FJOURNAL = {Journal of Algebra},
    VOLUME = {43},
      YEAR = {1976},
    NUMBER = {1},
     PAGES = {305--327},
      ISSN = {0021-8693},
   MRCLASS = {05C99 (17B20)},
  MRNUMBER = {441787},
MRREVIEWER = {M.\ Doob},
       DOI = {10.1016/0021-8693(76)90162-9},
       URL = {https://doi.org/10.1016/0021-8693(76)90162-9},
}

@article {ganzhinov,
    AUTHOR = {Ganzhinov, Mikhail and Szöllősi, Ferenc},
     TITLE = {Biangular lines revisited},
   JOURNAL = {Discrete Comput. Geom.},
  FJOURNAL = {Discrete \& Computational Geometry. An International Journal
              of Mathematics and Computer Science},
    VOLUME = {66},
      YEAR = {2021},
    NUMBER = {3},
     PAGES = {1113--1142},
      ISSN = {0179-5376,1432-0444},
   MRCLASS = {52C10 (05B30 94B05)},
  MRNUMBER = {4310607},
MRREVIEWER = {H.\ Kharaghani},
       DOI = {10.1007/s00454-021-00276-6},
       URL = {https://doi.org/10.1007/s00454-021-00276-6},
}

@article {bachoc,
    AUTHOR = {Bachoc, Christine and Vallentin, Frank},
     TITLE = {New upper bounds for kissing numbers from semidefinite
              programming},
   JOURNAL = {J. Amer. Math. Soc.},
  FJOURNAL = {Journal of the American Mathematical Society},
    VOLUME = {21},
      YEAR = {2008},
    NUMBER = {3},
     PAGES = {909--924},
      ISSN = {0894-0347,1088-6834},
   MRCLASS = {52C17 (90C22)},
  MRNUMBER = {2393433},
MRREVIEWER = {A.\ Florian},
       DOI = {10.1090/S0894-0347-07-00589-9},
       URL = {https://doi.org/10.1090/S0894-0347-07-00589-9},
}

@book {ebeling,
    AUTHOR = {Ebeling, Wolfgang},
     TITLE = {Lattices and codes},
    SERIES = {Advanced Lectures in Mathematics},
   EDITION = {Third},
      NOTE = {A course partially based on lectures by Friedrich Hirzebruch},
 PUBLISHER = {Springer Spektrum, Wiesbaden},
      YEAR = {2013},
     PAGES = {xvi+167},
      ISBN = {978-3-658-00359-3; 978-3-658-00360-9},
   MRCLASS = {11H31 (11T71 51F15)},
  MRNUMBER = {2977354},
MRREVIEWER = {Caleb\ McKinley\ Shor},
       DOI = {10.1007/978-3-658-00360-9},
       URL = {https://doi.org/10.1007/978-3-658-00360-9},
}

@article {lemmens,
    AUTHOR = {Lemmens, P. W. H. and Seidel, J. J.},
     TITLE = {Equiangular lines},
   JOURNAL = {J. Algebra},
  FJOURNAL = {Journal of Algebra},
    VOLUME = {24},
      YEAR = {1973},
     PAGES = {494--512},
      ISSN = {0021-8693},
   MRCLASS = {05C30 (52A40)},
  MRNUMBER = {307969},
MRREVIEWER = {J.\ J.\ Burckhardt},
       DOI = {10.1016/0021-8693(73)90123-3},
       URL = {https://doi.org/10.1016/0021-8693(73)90123-3},
}

@article {cao,
    AUTHOR = {Cao, Meng-Yue and Koolen, Jack H. and Munemasa, Akihiro and
              Yoshino, Kiyoto},
     TITLE = {Maximality of {S}eidel matrices and switching roots of graphs},
   JOURNAL = {Graphs Combin.},
  FJOURNAL = {Graphs and Combinatorics},
    VOLUME = {37},
      YEAR = {2021},
    NUMBER = {5},
     PAGES = {1491--1507},
      ISSN = {0911-0119,1435-5914},
   MRCLASS = {05C50 (05C22)},
  MRNUMBER = {4331063},
MRREVIEWER = {James\ McKee},
       DOI = {10.1007/s00373-021-02359-w},
       URL = {https://doi.org/10.1007/s00373-021-02359-w},
}

@article {ganzhinov2,
    AUTHOR = {Ganzhinov, Mikhail},
     TITLE = {Infinite families of optimal systems of biangular lines
              related to representations of {${\rm SL}(2,\Bbb F_q)$}},
   JOURNAL = {J. Combin. Theory Ser. A},
  FJOURNAL = {Journal of Combinatorial Theory. Series A},
    VOLUME = {192},
      YEAR = {2022},
     PAGES = {Paper No. 105656, 21},
      ISSN = {0097-3165,1096-0899},
   MRCLASS = {20G05 (05E10 20G40)},
  MRNUMBER = {4457402},
MRREVIEWER = {Antonio\ Cossidente},
       DOI = {10.1016/j.jcta.2022.105656},
       URL = {https://doi.org/10.1016/j.jcta.2022.105656},
}

@book {humphreys,
    AUTHOR = {Humphreys, James E.},
     TITLE = {Introduction to {L}ie algebras and representation theory},
    SERIES = {Graduate Texts in Mathematics},
    VOLUME = {9},
      NOTE = {Second printing, revised},
 PUBLISHER = {Springer-Verlag, New York-Berlin},
      YEAR = {1978},
     PAGES = {xii+171},
      ISBN = {0-387-90053-5},
   MRCLASS = {17Bxx},
  MRNUMBER = {499562},
MRREVIEWER = {I.\ P.\ Shestakov},
}

\end{document}